%% file: matching_dynamics_FCFS_v3.tex
\documentclass[11pt,reqno,oneside]{amsart}

\usepackage[title,header]{appendix}
\usepackage[letterpaper]{geometry}
\usepackage[utf8]{inputenc}
\usepackage[OT4]{fontenc}
\usepackage{amsmath}
\usepackage{amsthm}
\usepackage{amsfonts}
\usepackage{amssymb}
\usepackage{bm}
\usepackage{bbm}
\usepackage{color, soul}
\usepackage{stmaryrd}
\usepackage[table]{xcolor}
\usepackage{tikz}
\usepackage{hyperref}
\usepackage{mdframed}
\usepackage{enumerate}
\usepackage{longtable}

\theoremstyle{plain}
\newtheorem{theorem}{Theorem}[section]
\newtheorem{lemma}[theorem]{Lemma}

\newtheorem{corollary}[theorem]{Corollary}

\newtheorem{claim}[theorem]{Claim}
\theoremstyle{definition}
\newtheorem{definition}[theorem]{Definition}

\newtheorem{remark}[theorem]{Remark}

\newcommand{\numberthis}{%
  \refstepcounter{equation}\tag{\theequation}
}
\newcommand{\1}{\mathbbm{1}}

\DeclareMathOperator{\PP}{P}
\DeclareMathOperator{\EE}{E}

\DeclareMathOperator{\dom}{dom}
\DeclareMathOperator{\rx}{revx}

\newcommand{\cE}{\mathcal{E}}

\newcommand{\cP}{\mathcal{P}}

\newcommand{\cF}{\mathcal{F}}

\newcommand{\bbN}{\mathbb{N}}
\newcommand{\bbR}{\mathbb{R}}
\newcommand{\bbZ}{\mathbb{Z}}
\newcommand{\bbQ}{\mathbb{Q}}

\newcommand{\ignore}[1]{}

\newcommand{\Ind}{\mbox{}}
\newcommand{\IndI}{\mbox{}\qquad}
\newcommand{\IndII}{\mbox{}\qquad\qquad}
\newcommand{\IndIII}{\mbox{}\qquad\qquad\qquad}
\newcommand{\IndIIII}{\mbox{}\qquad\qquad\qquad\qquad}
\definecolor{DSgray}{cmyk}{0,0,0,0.7}
\definecolor{DSred}{cmyk}{0,0.7,0,0.7}
\newcommand{\Authornote}[2]{}

\DeclareFontFamily{U} {MnSymbolC}{} \DeclareSymbolFont{MnSyC} {U}
{MnSymbolC}{m}{n} \SetSymbolFont{MnSyC} {bold}{U} {MnSymbolC}{b}{n}
\DeclareFontShape{U}{MnSymbolC}{m}{n}{ <-6> MnSymbolC5 <6-7>
  MnSymbolC6 <7-8> MnSymbolC7 <8-9> MnSymbolC8 <9-10> MnSymbolC9
  <10-12> MnSymbolC10 <12-> MnSymbolC12}{}
\DeclareFontShape{U}{MnSymbolC}{b}{n}{ <-6> MnSymbolC-Bold5 <6-7>
  MnSymbolC-Bold6 <7-8> MnSymbolC-Bold7 <8-9> MnSymbolC-Bold8 <9-10>
  MnSymbolC-Bold9 <10-12> MnSymbolC-Bold10 <12-> MnSymbolC-Bold12}{}
\DeclareMathSymbol{\ftl}{\mathbin}{MnSyC}{202}
\DeclareMathSymbol{\ftu}{\mathbin}{MnSyC}{201}
\DeclareMathSymbol{\ftr}{\mathbin}{MnSyC}{200}

\newcommand{\ra}{\rightarrow}
\newcommand{\mbf}[1]{\ensuremath{\boldsymbol{#1}}}

\renewcommand{\bar}[1]{\overline{#1}}
\newcommand{\ttm}{\mathtt{m}}
\newcommand{\ttu}{\mathtt{u}}
\newcommand{\blambda}{\bar{\lambda}}

\newcommand{\tpi}{\tilde{\pi}}
\newcommand{\tPi}{\tilde{\Pi}}
\newcommand{\ttR}{\mathtt{R}}
\newcommand{\ttB}{\mathtt{B}}

\graphicspath{{./Figures/}}

\title{On Spatial Matchings: The First-in-First-Match Case}

\author{Mayank Manjrekar}
\thanks{\hspace{-1.5em} UT Austin, USA. E-mail:
  {\tt{mmanjrekar@math.utexas.edu}}\\This work is supported by award
  from the Simons Foundation (award number 197982) to the University
  of Texas at Austin.}

\date{}

\begin{document}

\begin{abstract}
  In this paper, we describe a process where two types of particles,
  marked by the colors \emph{red} and \emph{blue}, arrive in a domain
  $D$ at a constant rate and are to be matched to each other according
  to the following scheme.  At the time of arrival of a particle, if
  there are particles of opposite color in the system within a
  distance one from the new particle, then, among these particles, it
  matches to the one that had arrived the earliest. In this case, both
  the matched particles are removed from the system.  Otherwise, if
  there are no particles within a distance one at the time of the
  arrival, the particle gets added to the systems and stays there
  until it matches with another point later.  Additionally, a particle
  may depart from the system on its own at a constant rate, $\mu>0$,
  due to a loss of patience.  We study this process both when $D$ is a
  compact metric space and when it is a Euclidean domain, $\bbR^d$,
  $d\geq 1$.

  When $D$ is compact, we give a product form characterization of the
  steady state probability distribution of the process.  We also prove
  an FKG type inequality, which establishes certain clustering
  properties of the red and the blue particles in the steady state.
  When $D$ is the whole Euclidean space, we use the time-ergodicity of
  the construction scheme to prove the existence of a stationary
  regime.
\end{abstract}
\maketitle

\section{Introduction}
\label{sec:intro}
Let $D$ be a metric space, with complete metric $d$, and let $\lambda$
be a Radon measure defined over it.  Suppose for now that $D$ is
compact, so that $\lambda(D)<\infty$.  We study the time-evolution of
a continuous time stochastic Markov jump process $\{\eta_t\}_{t\geq0}$
whose state is defined by an ordered \emph{configuration} of two types
of points in $D$.  The two types are assigned to be the colors
\emph{red} and \emph{blue}, and referred in short by the letters
$\ttR$ and $\ttB$ respectively.  A configuration here refers to a
locally finite collection of points.  When $D$ is compact, a
configuration consists of finite number of points.

The process evolves over the space of ordered configurations on
$D\times\{\ttR,\ttB\}$ as follows.  New particles of each type arrive
according to an independent Poisson point process on $D\times \bbR^+$
with intensity $\lambda\times\ell$, where $\ell$ is the Lebesgue
measure on $\bbR^+$.  Suppose, for instance, that a red particle
arrives at time $t>0$, at location $x\in D$.  Then we look for the
first blue particle in the ordered sequence $\eta_{t-}$ whose distance
to location $x$ is less than $1$.  If there is such a particle,
$p\in\eta_{t-}$, the new state $\eta_t$ is obtained by removing the
particle $p$ from $\eta_{t-}$, while keeping the order of the
remaining elements fixed.  If there is no such particle, then the new
state is obtained by adding a particle, with location $x$ and mark
$\ttR $, to $\eta_{t-}$.  In this case, the order within the elements
of $\eta_{t-}$ is preserved and the new particle is placed at the end
of the sequence $\eta_{t-}$.  The arrival of a blue particle is
handled similarly.  Additionally, any particle in the configuration is
removed at a constant rate $\mu>0$, while preserving the order of the
remaining particles.  Note that if $\eta_0$ is the empty
configuration, then the ordering of particles in $\eta_t$ is simply
the order in which those particles have arrived in the system.  We
will call this the First-in-first-match (FIFM) spatial matching
process.

Figure~\ref{fig:illustration} gives an illustration for the above
dynamics.

\begin{figure}[ht]
  \centering \def\svgwidth{0.8\textwidth}
  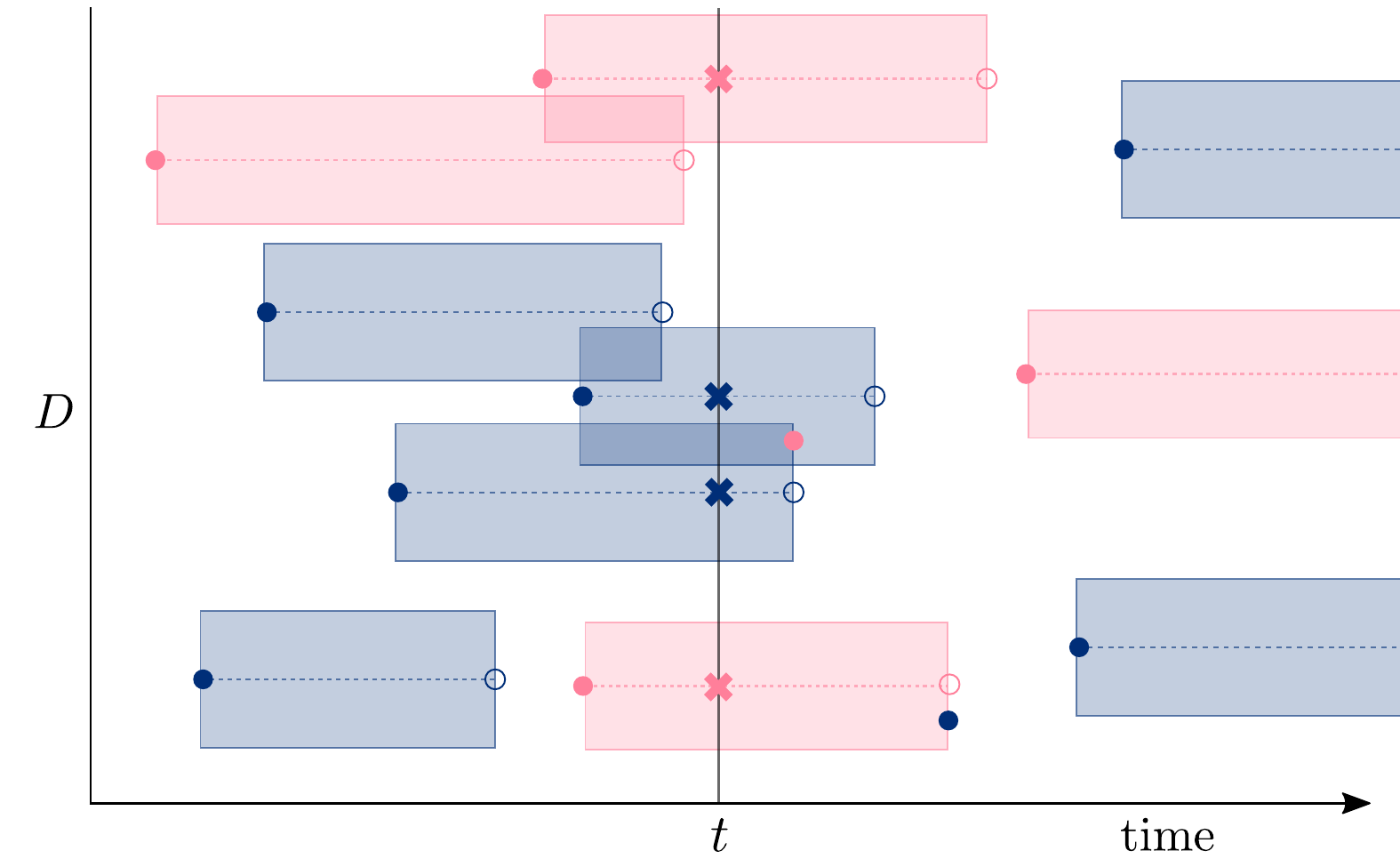
  \caption{An illustration of the FIFM spatial matching process.  The
    vertical dimension represents the set $D$.  The rectangles
    represent the lifetimes of particles in the system -- so, the
    vertical dimension of the rectangle represents the spatial range
    of interaction of a particle, a solid disk to the left of a
    rectangle represent its arrival, and a hollow circle at the right
    represents its departure.  The set of particles present in
    $\eta_t$ are marked by crosses; these are those particles whose
    rectangles intersect the ``vertical line'' at time
    $t$.}\label{fig:illustration}
\end{figure}

\subsection{Motivation and previous work}
\label{sec:motivation}
The motivation for studying this problem comes from modern
shared-economy markets, where individuals engage in monetized exchange
of goods that are privately owned in a via peer-to-peer marketplace.
Examples of such marketplaces include ride-sharing networks, such as
Uber or Lyft, and renewable energy networks with distributed
generation of power.  Here, consumers and producers can be viewed as
individuals distributed in an abstract space, who engage in a
transaction with others in close proximity.  The abstract space could
model factors such as location, product preferences, price and
willingness to pay, etc. In the example of ride-sharing network, the
position of an individual would correspond to its physical location,
while in a renewable energy network, the position could model a
combination of physical location and price.  In this study, our goal
is to comment on the spatial distribution of individuals in the
long-run, under a well-defined matching scheme such as the one
described in the introduction.

The underlying dynamics in our model can be viewed from a queuing
theoretic viewpoint.  Most queuing theoretic models study systems
where there is an inherent asymmetry between customers and servers.
Customers are usually transient agents that arrive with some load, and
depart on being processed.  Servers meanwhile are present during the
whole life-time of the study of the stochastic process, and serve the
customers according to a given policy.  In the literature, there are
only a few examples of queuing systems where customers and servers are
treated as symmetric agents that \emph{serve} each other.  The
double-ended queuing model discussed in \cite{kashyap1966double}
studies a model for a taxi-stop where taxis and customers arrive
independently according to two Poisson arrivals.  If a taxi (or a
customer) arrives at the taxi-stop and finds a waiting customer (taxi)
waiting, then it matches instantly, using say a first-come-first-serve
(FCFS) policy, and both agents depart.  Otherwise, the taxi (customer)
waits until it is matched with a customer (taxi) that arrives later.

The FCFS bipartite matching model that was introduced in
\cite{caldentey2009fcfs}, and later studied in some generality in
\cite{Adan2018}, is another such model.  In this model, the customers
and servers belong to finite sets of types, $C$ and $S$ respectively,
which determine whom they can be matched to.  The compatibility of
matches between the various types of customers and servers is
expressed in terms of a bipartite graph $G=(C, S,\cE)$, where
$\cE\subset C\times S$.  The process, $\{\eta_t\}_{t\in \bbN}$, is the
ordered list of unmatched customers and servers arriving before time
$t\in \bbN$.  At each time $t\in\bbN^+$, one customer, $c_t\in C$, and
one server, $s_t\in S$, arrive to the system.  Here,
$\{c_t\}_{t\in \bbN^+}$ and $\{s_t\}_{t\in\bbN^+}$ are independent
sequences of i.i.d.  random elements of $C$ and $S$, with
distributions $\alpha$ and $\beta$, respectively.  $\eta_{t}$ is
obtained from $\eta_{t-1}$, $c_t$ and $s_t$ by matching $c_t$ and
$s_t$ from elements in $\eta_t$, if possible, using the FCFS policy,
and removing the matched pairs.  This model is called the \emph{FCFS
  bipartite matching model}.  In this series of works
(\cite{caldentey2009fcfs,adan2012exact,Adan2018}), the authors derive
a product form distribution for the steady state under the so-called
\emph{complete resource pooling} condition:
\begin{align*}
  \alpha(A)< \sum_{(x,y)\in A\times S\subseteq E}\beta(y),\ \forall
  A\subsetneq C,
\end{align*}
or equivalently,
\begin{align*}
  \beta(A)< \sum_{(x,y)\in C\times A\subseteq E}\alpha(x),\ \forall
  A\subsetneq S.
\end{align*}
The authors also provide expressions for performance measures in the
steady state, such as the matching rates between certain type of
pairs, and waiting times of agents.  These expressions are
computationally hard to evaluate, owing to the hardness in computing
the normalizing constant in the product form distribution.

We now briefly discuss variants of the FCFS bipartite matching model.
Bu{\v{s}}i{\'c} et.al (\cite{buvsic2013stability}) generalize the
bipartite matching model by dropping independence of arriving types
and considering other matching policies.  B\"uke and Chen
\cite{buke2017fluid} study a model where the matching policy is
probabilistic.  In their model, when a customer (server) arrives in a
system, it looks at the possible matches and independently of
everything else, selects one using a probability distribution.  There
is also a positive probability of not finding any suitable server
(customer), in which case it starts to wait for a compatible server
(customer).  They also consider models where the users are impatient
and may depart if they are not matched by a certain time.  An exact
analysis of these models becomes quite intractable and in
\cite{buke2017fluid} the authors study the fluid and diffusive scaling
approximations of these systems.

The model we consider in this paper is essentially a continuous time
and continuous space version of the model studied in \cite{Adan2018},
with the added feature that particles may also depart on their own due
to a loss of patience.  This spatial matching model is related to the
FCFS bipartite matching model in the following sense.  Just like the
classes of customers and servers in that model, in our model, we still
have two classes, the red and blue particles; each particle has a
location in $D$ which is akin to the types within a class; and two
particles are supposed to be compatible in the sense of the FCFS
bipartite matching model, if they are within a distance one from each
other.

In this paper, in the case when $D$ is compact, we derive a product
form characterization of the steady state distribution of the process
we consider.  The analysis needed to obtain this product form
distribution is an extension of the analysis in \cite{Adan2018} to the
continuum.  We guess the reversed process and the steady state, and
then check the local balance conditions to get the product form
distribution in Theorem~\ref{thm:invMeasure}.

The two particle Widom-Rowlinson (WR) model is related to the
distribution of the unordered configuration, $\tilde\eta$, in the
steady state of our model.  The one and the two particle WR models
were defined in \cite{widom1970new} as a mathematical model for the
study of liquid-vapor phase transition in physical systems.  The two
particle WR model is a point process that consists of two types of
particles, $\ttR$ and $\ttB$, as in the steady state of our process.
This model on a compact domain can be described as the union of two
Poisson point processes with intensities $\lambda_{\ttR }$ and
$\lambda_{\ttB}$, conditioned on the event that there are no two
particles of opposite types within a distance one from each other. On
the infinite Euclidean domain, $\bbR^d$, it is defined as a Markov
random field (see \cite{van2000markov}), with Papangelou conditional
intensity $\varphi((x,\ttR),\eta)=\lambda_\ttR\1(d(x,\eta^\ttB)>1)$
and $\varphi((x,\ttB),\eta)=\lambda_\ttB\1(d(x,\eta^\ttR)>1)$, where
$\eta^\ttR$ and $\eta^\ttB$ are the collection of red and blue points
of $\eta$.  The single particle WR model is obtained by marginalizing
over one of the particles in the two particle WR model.  The two
particle WR model is interesting as it is the first continuum Markov
random process where a phase transition has been rigorously
established.  It was basically established that, in the phase diagram,
along the line $\lambda_{\ttR }=\lambda_{\ttB}=\lambda$, symmetry is
broken when $\lambda$ is large enough.  This was first shown in
\cite{ruelle1971existence} using an adaptation of Peierl's argument.
\cite{chayes1995analysis,giacomin1995agreement} independently give
modern self-contained proofs of this phenomenon using percolation
based arguments.  The key ingredient in the proof in
\cite{chayes1995analysis} is the observation that when we disregard
the types of the particles, the resulting model, called the Gray WR
model, is the continuum version of the random cluster model.  The Gray
model, in particular, satisfies the FKG property (with the usual
lattice structure) and the corresponding positive association
inequalities, that are crucial in showing existence of a percolation
thresholds, and consequently a phase transition in this model.

Contrary to the two particle WR model, we believe that in our model,
the unordered collection of the points in the steady state is not a
Markov points process. The definition of a Markov point process (see
Chapter 2 of \cite{van2000markov}) requires, first, that there exists
a symmetric reflexive relation $\sim$ over the domain, and second,
that the Papangelou conditional intensity at a point $x$ depend on the
configuration $\eta$ only through the points in $\eta$ that are
related to the $x$ by $\sim$. In our case, we are unable to show the
existence of such a symmetric relation. However, we are able to show
that the Papangelou conditional intensity at a point depends on the
clusters of overlapping unit balls that intersect with the unit ball
around the point.

In spite of this limitation, we show that the point process satisfies
an FKG lattice property, with a specific lattice structure, similar to
one satisfied by the two particle WR model (see Section 2 of
\cite{chayes1995analysis}). The lattice structure as follows: we say
$\eta>\eta'$ if and only if $\eta^{\ttR}\supset{\eta'}^{\ttR}$ and
$\eta^{\ttB}\subset{\eta'}^\ttB$. The resulting positive association
inequality can be used in conjunction with the results in
\cite{blaszczyszyn2014comparison}, to prove that the points of the
same type are weakly-super Poissonian. This is interesting since any
exact analysis of the clustering of the steady state from its product
form distribution is prohibitively hard as there is no closed
form expression for its normalizing constant. In fact, in discrete
systems, it is a $\sharp P$-complete problem to compute the
normalizing constant (\cite{adan2012exact}).

In this paper, we also consider the same matching dynamics in the
infinite Euclidean domain $\bbR^d$.  In this regime, using coupling
from the past based arguments, we give a formal
definition and a construction of the process, and show that there
exists a stationary regime for this process. The existence of the
stationary regime is obtained using certain coupling from the past
ideas that were developed in \cite{baccelli2017mutual}.

In the following sections, we will discuss the notation required to
formally define our model.  Other notation will be required as we go
along -- see Appendix~\ref{sec:tabl-nota} for a table of notation.  We
begin Section~\ref{sec:form-desc} with the formal definition of the
model in a compact domain. In Section~\ref{sec:exis-uniq-stea-stat},
we give a coupling based argument that the steady state exists and is
unique, and in Section~\ref{sec:prod-form-char}, we present the
product form distribution for this steady state. Then, in
Section~\ref{sec:clus-prop-fkg}, we introduce and prove the FKG
lattice property satisfied by the unordered version product form
distribution.  We then proceed to study the model in the infinite
Euclidean domain, in Section~\ref{sec:defi-eucl}. In
Section~\ref{sec:stat-eucl}, we give a construction and in
Section~\ref{sec:CFTP}, we give the construction of the stationary
regime.

\subsection{Notation}
\label{sec:notation}
\hypertarget{def:notation}{}
Let $S$ be any metric space, endowed with a Radon measure,
$\lambda_S(\cdot)$.  We use $M(S)$ to denote the space of simple
counting measures on $S$.  There is one-to-one correspondence between
$M(S)$ and the space of locally finite collection of points in $S$.
$M(S)$ is equipped with the $\sigma$-algebra $\cF$ generated by the
maps $\gamma\mapsto\gamma(B)$, $B\in \mathcal{B}(S)$, where
$\mathcal{B}(S)$ is the Borel $\sigma$-algebra on $S$.  In our
presentation, we will frequently abuse notation and use the same
variable to denote both an element of $M(S)$ and its support, which is
a subset of $S$.

Every particle in our model also carries information about its
patience.  To encode this, we need the notion of a marked counting
measure.  A marked simple counting measure on $S$, with marks in a
space $K$, is a locally finite simple counting measure $\gamma$ on
$S\times K$ such that its projection $\gamma(\cdot\times K )$ is an
element of $M(S)$.  We denote the space of all such measures by
$M(S,K)$.

We will also require the definition of the space of locally-finite
totally-ordered collection of points in the space $S$, $O(S)$.  For
any $\xi\in O(S)$, the order within the elements of $\xi$ will be
denoted by $<_\xi$.  The order will be used to indicate the priority
of the particles when matching with other particles.  So, if the state
of the system is $\xi\in O(S)$ and an incoming point $x$ is compatible
with both $y_1<_\xi y_2$, then it prefers $y_1$ over $y_2$.  $O(S)$
has a natural projection onto $M(S)$, obtained by dropping the order
within its elements -- for any $\xi\in O(S)$, the unordered collection
is denoted by $\tilde\xi$. For compact $S$, $O(S)$ may be canonically
identified with
$\sqcup_{n=0}^\infty \{x\in S^n:x_i\neq x_j,\forall 1\leq i<j\leq
n\}$.  Finally, the space of totally-ordered marked locally-finite
collection of points, with marks in $K$, will be denoted by $O(S,K)$.

For any $\gamma\in M(S)$ (or $O(S)$), we will use the notation
$|\gamma|$ to denote the number of elements in $\gamma$, i.e.,
$|\gamma|=\gamma(S)$.

\hypertarget{def:oppositecolor}{}\hypertarget{def:colorset}{} As
mentioned in the introduction, the symbols $\ttR $ and $\ttB$ will be
used to denote the types red and blue respectively.  Moreover, we will
let $\mbf{C}=\{\ttR ,\ttB\}$, and let a line over a color denote the
opposite color, i.e., $\bar\ttR=\ttB$ and $\bar\ttB=\ttR$.

\section{First-in-First-Match Matching Process on Compact Domains}
\label{sec:form-desc}
\hypertarget{def:N}{}\hypertarget{def:formaldesc}{}\hypertarget{def:p-x-b-x}{}\hypertarget{def:gamma-sup-x}{}

In this section, we first give a formal definition of the process on a
compact domain.  Let $D$ be a compact metric space with a Radon
measure $\lambda$.  The state of the process will contain information
about the location, color and the order of arrival of the particles
present in the system.  Thus, the state space will be the set of
totally-ordered collection of particles, with location in $D$ and with
marks in the set $\mbf{C}=\{\ttR,\ttB\}$, namely $O(D,\mbf{C})$.  The
order represents the order of arrival of particles into the system,
and hence represents their priority when two particles are in
contention to be matched to the same particle.

We will require the following notation to describe the evolution of
the process.  For a point $x\in D\times\mbf{C}$, we denote the
projection onto $D$ by $p_x$ and denote the projection onto $\mbf{C}$
by $c_x$.  For any point $x\in D\times \mbf{C}$, we denote the set of
incompatible points of opposite color by
$N(x):=B(p_{x},1)\times \{\bar{c_x}\}$, where $B(z,r)$ denotes the
ball of radius $r$ centered at $z$.  For any subset
$A\subseteq D\times \mbf{C}$, we set $N(A):=\cup_{x\in A }N(x)$.  Let
$\blambda:=\lambda\otimes m_c$, where $m_c$ is the counting measure on
$\mbf{C}$.  For any $\gamma\in O(D,\mbf{C})$ and $x\in \gamma$, let
$\gamma^x$ be the element of $O(D,\mbf{C})$ formed by
$\{y\in \gamma: y<_\gamma x\}$ ordered as in $\gamma$.  Further, if
$\gamma$ is represented as a list $(x_1,\ldots,x_n)$, then for any
$i$, $1\leq i\leq n$, we set
$\gamma_1^{i-1} :=\gamma^{x_i} =(x_1,\ldots,x_{i-1})$.  The region of
highest priority of a particle $x$ in $\gamma$, denoted $W_{\gamma,x}$
(or just $W_x$ if the context is clear), is defined to be the set
$N(x)\backslash N(\gamma^{x})$.

Let us recall the description of the dynamics in terms of the above
notation.  We consider a Markov jump process,
$\{\eta_t\}_{t\in\bbR}\subset O(D,\mbf{C})$.  Suppose that a new
particle, $y\in D\times\mbf{C}$, arrives at time $t$.  If the ordered
collection $\eta_t\cap N(y)$ is non-empty, then the particle $y$
\emph{matches} to the lowest-ranked particle in this set and the
matched particle is removed from $\eta_t$.  Equivalently, $y$ matches
to $x\in \eta_t$ if and only if $y\in W_{\eta_t,x}$.  Otherwise, $y$
is added to the ordered set $\eta_t$ at the end, so that $y>x$ for all
$x\in \eta_t$, while the order among the elements of $\eta_t$ is
preserved.  Additionally, independent of everything else, particles
may depart on their own when they lose patience at rate $\mu>0$.

This description fixes the form of the generator of the process, which
is given by
\begin{align}
  \label{eq:generator}
  Lf(\eta):=
  &\sum_{x\in\eta}\left(\mu+\blambda(W_x)\right) [f(\eta\backslash
    x) -f(\eta)]+\int_{D\times\mbf{C}}\1(x\notin
    N(\eta))[f(\eta,x)-f(\eta)]\blambda(dx),
\end{align}
where $f$ is a measurable function defined over $O(D,\mbf{C})$. In the
following we give an explicit construction of a process which will
serve as the formal definition of our process. It can be easily
verified that, if $\{\eta_t\}$ is the constructed process, then
$$\lim_{t\ra 0+}\frac{1}{t}\EE
[f(\eta_t)-f(\eta_0)|\eta_0]=Lf(\eta_0),$$ for any bounded continuous
function $f$ over $O(D,\mbf{C})$ (we refrain from identifying the full
domain of the generator). For us, the form of the generator will be
useful in characterizing the stationary distribution, while the
explicit construction will be useful later in the construction of the
process on $\bbR^d$.

\hypertarget{def:Phi}{} Let $\Phi$ be a Poisson point process on
$D\times \bbR^+$, with i.i.d.  marks in $\mbf{C}\times \bbR^+$.  The
intensity of the point process is $2\lambda\otimes \ell$, where $\ell$
is the Lebesgue measure on $\bbR^+$.  Both the marks are independent,
with the color uniformly distributed and the other mark is an
exponential random variable with parameter $\mu$.  Let
$\eta_0\in O(D,\mbf{C})$ be the initial state of the system at time
$0$.  Each point in $x\in\Phi$ is represented by four coordinates
$(p_x,b_x,c_x,w_x)$, with $p_x\in D$, $b_x,w_x\in \bbR^+$ and
$c_x\in\mbf{C}$.  $p_x$ denotes the spatial position of the point $x$,
$b_x$ denotes the time of its arrival, $c_x$ denotes its color and
$w_x$ denotes its patience.  The following display presents an
algorithm for the construction of the process on compact domains.
\\
\begin{mdframed}
  \begin{itemize}
  \item \textbf{Data:}
    \begin{enumerate}[a.]
    \item $\Phi$: A realization of the arrivals.
    \item $\eta_0$: A realization of the initial condition.
    \item $t\in\bbR^+$: End time of simulation.
    \end{enumerate}
  \item \textbf{Result:} $\eta_t$: The final state of the system as
    time $t$.  \end{itemize}
  \begin{enumerate}
  \item Set $t_{old}=0$.
  \item For each $x\in \eta_0$, assign i.i.d.  marks $w_x$, that are
    exponentially distributed with parameter $\mu$.
  \item Set
    $t_{new}= \min (\inf\{b_x:x\in \Phi_{(t_{old},\infty)}\}
    ,\inf\{w_x:x\in\eta_{t_{old}}\})$.  If $t_{new}> t$, quit and
    return $\eta_{t_{old}}$.
  \item If $t_{new}$ is due to arrival of a new particle (first
    infimum):
    \begin{itemize}
    \item Let the particle be $x$.
    \item If there is a particle of opposite color in $B(p_x,1)$:
      \begin{itemize}
      \item Match to the first particle of opposite color in
        $\eta_{t_{old}}\cap B(p_x,1)$ and remove that particle.  This
        gives $\eta_{t_{new}}$.
      \end{itemize}
    \item Else:
      \begin{itemize}
      \item Add the particle to the end of $\eta_{t_{old}}$ to give
        $\eta_{t_{new}}$
      \end{itemize}
    \end{itemize}
  \item Else if $t_{new}$ is due to a particle $x\in\eta_{t_{old}}$
    losing patience (second infimum), then remove this particle to
    yield $\eta_{t_{new}}$.
  \item Set $t_{old}=t_{new}$.  Go to Step 3.
  \end{enumerate}
\end{mdframed}

\subsection{Existence and Uniqueness of a Stationary Regime}
\label{sec:exis-uniq-stea-stat}
In this section, we look at the stationary regime of the process on a
compact domain $D$, defined in Section~\ref{sec:form-desc}.  We first
show that there is a unique stationary measure for this process, and
in the subsequent sections give a product form characterization.

It can be seen that if in the model, we have $\mu=0$, the process
$\{\eta_t\}$ does not have a stationary regime.  Indeed, in this case,
starting from $\eta_0=\emptyset$,
$|\eta_t|\geq |\Phi(D\times[0,t],\ttR)-\Phi(D\times[0,t],\ttB)|$.  The
process on the right-hand side does not have a stationary regime. So,
we need to assume $\mu>0$. It is easy to argue then, using a standard
coupling from the past or a Lyapunov technique, that there is a unique
stationary distribution.  For the sake of completeness, we give a
coupling from the past construction of a stationary regime and prove
that it is unique.

Suppose we have a bi-infinite time-ergodic Poisson point process
$\Phi$ on $D\times \bbR$, with marks in $\mbf{C}\times \bbR^+$, where
the first coordinate is the color of the particle and the second
coordinate is the time the particle is the patience, as in the
construction in Section~\ref{sec:form-desc}.  We define the notion of
a \emph{regeneration time} of the Poisson point process $\Phi$ as
follows.  A time $t\in\bbR$ is called a regeneration time if for all
$x\in \Phi$, with $b_x\leq t$, we have $t-b_x>w_x$.  That is, there is
no possibility that a particle arriving before $t$ survives beyond
time $t$.  For any process, $\{\eta_s^r\}_{s\geq r}$, started with
empty initial conditions at time $r$, and driven by the process
$\Phi$, we note that $\eta_s^r=\eta^t_s$, for all $s\geq t$ and all
regeneration times $t\geq r$.  Therefore, a stationary regime exists
if we can show the existence of a sequence of regeneration times that
diverge to $-\infty$ almost surely.  This is an instance of a coupling
from the past scheme.  The following lemma provides such a sequence of
regeneration times.

\begin{lemma}\label{lem:boundedRegeneration}
  Under the setting of this section, there are infinitely many regeneration
  times in the list $0,-1,-2,\ldots$, almost surely.
\end{lemma}
\begin{proof}
  Let us find the probability of the event, $A_0$, that $0$ is a
  regeneration time.  We have:
  \begingroup
  \allowdisplaybreaks
  \begin{align*}
    \PP(A_0)&=\PP(w_x<-b_x,\forall x\in \Phi, b_x<0)\\
            &=\EE\prod_{x\in \Phi}\1(w_x<-b_x,b_x<0)\\
            &=\EE\lim_{s\ra\infty}e^{-s\int\1(w_{x}\geq -b_x)\Phi(dx)}\\
            &\geq
              \limsup_{s\ra\infty}\exp\left(\int_{D\times\mbf{C}}\int_{\bbR^-\times\bbR^+}\left(e^{-s\1(w\geq-b)}-1
              \right)\mu e^{-\mu w}dwdb\lambda(dp)\right)\\
            &=\limsup_{s\ra\infty} \exp\left(2\lambda(D)\int_{\bbR^+}(e^{-s}-1)e^{-\mu b}db\right)\\
            &=\limsup_{s\ra\infty}\exp\left(2\lambda(D)(e^{-s}-1)/\mu)\right)\\
            &=\exp\left(-2\lambda(D)/\mu\right) >0,
  \end{align*}
  \endgroup
  where in the third equation we have used the Fatou's lemma and the
  Laplace transform formula for Poisson point processes.  Now, let
  $A_n$ be event that $-n$ is a regeneration time.  If $\theta_{t}$ is
  a time-shift operator, we have $A_n=\theta_{-n}A_0$.  By
  time-ergodicity of $\Phi$, $A_n$ must occur infinitely often, almost
  surely.  Thus, there are infinitely many regeneration times in the
  list $\{0,-1,-2,\ldots\}$.
\end{proof}

The uniqueness of a stationary regime can also be show using a
coupling argument.  We only give an outline of this procedure here.
Suppose we consider two stationary measures of the process.  Let
$\eta^1_0$ and $\eta^2_0$ be realizations of these two states.  For
large enough $n$, the probability that $|\eta^1_0|>n$ and
$|\eta^2_0|>n$ is less than $\epsilon>0$.  Conditioned on this event
we may couple the processes in a time $T$, using the coupling scheme
from the previous lemma. Note that with such a scheme, we have
$\EE T<\infty$.  Thus, the total variation distance
$d_{TV}(\eta^1_t,\eta^2_t)\leq \epsilon + \EE T/t$, using the coupling
and the Markov inequalities.  Since
$d_{TV}(\eta^1_0,\eta^2_0)=d_{TV}(\eta^1_t,\eta^2_t)$, we must have
that $\eta^1_0$ must be equal in distribution to $\eta^2_0$.

In the next section, we present a product form characterization of
this steady state distribution.  To do this, the key step is to
construct the reversed process.

\subsection{Product Form Characterization of the Steady State}
\label{sec:prod-form-char}
Let $\Phi$ be the driving Poisson point process on $D\times\bbR$, with
i.i.d marks in $\mbf{C}\times\bbR^+$, that is given as data as defined
in the coupling from the past construction in
Section~\ref{sec:exis-uniq-stea-stat}.

\hypertarget{def:matching-func}{} Lemma~\ref{lem:boundedRegeneration}
implies that there exists a unique bi-infinite spatial matching
process that is driven by $\Phi$.  We can thus define a (random)
matching function, $m:\Phi\ra D\times\bbR\times\mbf{C}$, such that
\begin{align*}
  m(x)=
  \begin{cases}
    (p_x,b_x+w_x,c_x) &\textrm{if $x$ exits on its own,}\\
    (p_y,b_y,c_y) &\textrm{if $x$ matches to $y\in \Phi$.}
  \end{cases}
\end{align*}
Conversely, $m$ stores all the information necessary to build the
process $\{\eta_t\}_{t\in\bbR}$.  Indeed, the state of the
spatial matching process we are interested in is given by
\begin{align*}
  \eta_t=((p_x,c_x):x\in \Phi, b_x\leq t <b_{m(x)}),
\end{align*}
where the list is ordered according to the birth-times, $b_x$.

\hypertarget{def:matching-marks}{} To get a handle on the stationary
distribution of this process, we shall create its reversed process.
Taking inspiration from \cite{Adan2018}, we will include some
additional data in the state of the system that will simplify the
description of the reversed process.  We shall consider a process that
we call the \emph{backward detailed process} generated by $\Phi$ and
$m$.  This process contains unmatched and matched particles in its
state, and we distinguish these types by using marks ``$\ttu$'' or
``$\ttm$'' respectively.  For any particle $x$ in the state, $s_x$
will refer to this mark.

\hypertarget{def:hat-eta_t}{} For $t\in\bbR$, let
$$T_t := \min \{b_x:x\in \Phi, b_x\leq t < b_{m(x)}\}$$ be the time of
arrival of the earliest among the unmatched particles at time $t$.
Let
$$\Gamma_{\ttu}:=\{(p_x,b_x,c_x,\ttu):x\in \Phi, b_x\leq t < b_{m(x)}\}$$
be the set of (location, arrival-times and colors of) unmatched particles
in $[T_t,t]$.  Let
\begin{align*}
  \Gamma_\ttm&:=\{(p_{x},b_{m(x)},c_{x},\ttm):x\in \Phi,b_{x} \leq t, T_t\leq
               b_{m(x)}\leq t\}\\
             &=\{(p_{m(x)},b_{x},c_{m(x)},\ttm):x\in \Phi, b_{m(x)}\leq t,
               T_t\leq b_x\leq t, c_x\neq c_{m(x)}\}\\
  &\IndII\cup\{(p_{x},b_{m(x)},c_{x},\ttm):x\in \Phi,b_{x} \leq t, T_t\leq
               b_{m(x)}\leq t, c_x=c_{m(x)}\},
\end{align*}
be the set of so-called \emph{matched and exchanged} particles that
are present in $[T_t,t]$. In the last expression, the first set of
elements corresponds to particles that arrive in the relevant
interval, $[T_t,t]$, and are matched by the time $t$; but instead of
recording their positions and types, we record that of their
matches. The second set of elements in that expression corresponds to
particles arrive before $t$, that depart on their own in the time
interval $[T_t,t]$; we record the time at which they depart.

Finally, define the backward detailed process, $\hat \eta_t$, be the
list $((p_x,c_x,s_x):x\in \Gamma_\ttu\cup \Gamma_\ttm)$, ordered
according to the values of $b_{(\cdot)}$. Clearly, the original
process $\eta_t$ can be obtained from $\hat\eta_t$ by removing the
particles with marks $s_x=\ttm$.  Notice that if $|\hat\eta_t|>0$, the
first element in $\hat\eta_t$, denoted by $x_1$, always satisfies
$s_{x_1}=\ttu$.

The backward detailed process, $\hat\eta_t$, is a stationary version
of a Markov process. A valid state of this Markov process is any
finite list of elements, $(x_1,\ldots,x_n)$ from the set
$D\times\mbf{C}\times \{\ttu,\ttm\}$ that satisfies the following
definition.
\begin{definition}[Definition of a valid state of
  $\hat\eta_t$] \label{cond:vali-stat}
  We say that a finite list of elements $(x_1,\ldots,x_n)$, with
  $n\in\bbN$ and $x_i\in D\times\mbf{C}\times\{\ttm,\ttu\}$, is a
  \emph{valid} state of $\hat\eta_t$ if the following three
  conditions are satisfied:
  \begin{enumerate}
  \item $s_{x_1}=\ttu$, if $n\geq 1$.
  \item For all $1\leq i,j\leq n$, $s_{x_i}=s_{x_j}=\ttu$ and
    $d(p_{x_i},p_{x_j})\leq 1$ implies that $c_{x_i}= c_{x_j}$.
\item For all $1\leq i<j\leq n$, $s_{x_i}=\ttu$, $s_{x_j}=\ttm$ and
  $d(p_{x_i},p_{x_j})\leq 1$ implies that $c_{x_i}= c_{x_j}$.
\end{enumerate}
\end{definition}
Condition 2 in the above definition essentially states that there
cannot be a compatible unmatched pair in a valid state.  This
condition is equivalent to the condition that
$$\{y\in x_1,\ldots x_n:s_y=\ttu\}\cap N(\{y\in x_1,\ldots
x_n:s_y=\ttu\})=\emptyset.$$ Condition 3 cannot be violated, since
otherwise the particle whose matched and exchanged pair is $x_j$ could
instead have matched to $x_i$ that arrives earlier. This condition is
equivalent to the condition that for all $1\leq j\leq n$,
$$s_{x_j}=\ttm\implies x_j\notin N(\{y\in x_1,\ldots,
x_j:s_y=\ttu\}.$$ Any valid state can be achieved by the process
$\hat\eta_t$ in finite time with positive probability. Indeed,
starting from the empty state, a valid state, $\hat\eta$, can result
from empty state if the arrivals occur in the order listed in
$\hat\eta$, with appropriate patience so that the particles in
$\hat\eta$ marked $\ttu$ survive until time $t$, and the particles in
$\hat\eta$ marked $\ttm$ exit on their own before the next arrival.

Transitions for $\hat\eta_t$ occur at the time of arrival of a new
particle or at the event of a voluntary departure. At the time of a
new arrival, we match and exchange the particles in the list
$\hat\eta_t$, and at the time of a departure, we put the departing
particle at the end of the list $\hat\eta_t$, while updating the mark
to $\ttm$. Below, we describe the transitions and transition rates of
this Markov process in detail.

The transitions and transition rates for $\hat\eta_t$ are as follows:
Let $\hat\eta=(x_1,\ldots,x_n)$, $n\in \bbN$, be a valid state.
\begin{enumerate}
\item A particle $x_i\in\hat\eta$, with $s_{x_i}=\ttu$, loses
  patience:  This occurs at rate $\mu$. In this case, the new state is
  obtained by removing the $x_i$ and inserting
  $(p_{x_i},c_{x_i},\ttm)$ at the end of the list $\hat\eta$.
  Additionally, we need to prune leading matched and exchanged
  particles from $\hat\eta$ to obtain the new state.
\item A new particle $y=(p_y,c_y)$ arrives and is matched to a
  particle $x_i\in \hat\eta$, with $d(p_{x_i},p_y)\leq 1$ and
  $c_{x_i}\neq c_y$: This occurs at rate
  $\blambda(dy) \1(y\in W_{x_i})$.  The new state is obtained by
  matching and exchanging the appropriate pair, and then pruning the
  leading matched and exchanged particles.
\item A new particle $y$ arrives and there is no particle of opposite
  color within a distance $1$ from it:  This occurs at rate
  $\blambda(dy) \1(y\notin N(\hat\eta))$.  The new state is the one
  obtained by adding this new particle to the end of the list as an
  unmatched particle.
\end{enumerate}

\hypertarget{def:check-eta_t} We now guess the time-reversed version
of the backward-detailed process, and obtain its transition rates.
The following construction will be useful in doing this.  Consider a
dual process $\check\eta_t$, that we call the \emph{forward detailed
  process}.  It is defined as follows: for $t\in\bbR$ let
$$Y_t:=\max\{b_{m(x)}:x\in \Phi, b_x\leq t< b_{m(x)}\},$$
be the latest time at which all unmatched particles at time $t$ are
matched or exit. Let
\begin{align*}
  \Xi_\ttm&:=\{(p_x,b_{m(x)},c_x,\ttm):x\in \Phi, b_x\leq t<b_{m(x)}\}\\
  &=\{(p_{m(x)},b_x,c_{m(x)},\ttm):x\in \Phi, b_{m(x)}\leq t <b_{x},
  c_x\neq c_{m(x)}\}\\
  &\IndII\cup\{(p_x,b_{m(x)},c_x,\ttm):x\in\Phi, b_x\leq t<b_{m(x)},c_x=c_{m(x)}\},
\end{align*}
 be the
matched and exchanged particles corresponding to the particles that
are born before time $t$, but have not been removed from the system by
time $t$.  Let
\begin{align*}
\Xi_{\ttu}:=\{(p_x,c_x,b_{x},\ttu):x\in \Phi, t<b_x<Y_t,t<b_{m(x)}\},
\end{align*}
be the particles in the relavant interval $(t,Y_y)$, whose match
arrives after time $t$.  Now, let
$\check\eta_t= ((p_x,c_x,s_x):x\in \Xi_u\cup\Xi_{\ttm})$,
ordered according to the values $b_{(\cdot)}$.  Thus, the last element
in the list $x_{|\check\eta|}$, always has
$s_{x_{|\check\eta|}}=\ttm$. For motivations for these definitions, see \cite{Adan2018}.

Under our construction, using the bi-infinite Poisson point process
$\Phi$, the process $\{\check\eta_t\}_{t\in\bbR}$ is a stationary
process.  In fact, it is a stationary version of a Markov process,
since all the arrivals and deaths are Markovian.  The transitions and
transition rates are defined in detail in
Appendix~\ref{sec:calculations}.

The underlying idea in obtaining the product form distribution is the
following.  For any list of elements $\gamma$, let
$\rx(\gamma)$ be the list of elements in $\gamma$ written in
the reverse order, with the marks $\ttu$ and $\ttm$ flipped.  Then we
claim that, a version of time-reversal of the backward-detailed
process $\{\hat\eta_t\}_{t\in\bbR}$, is given by
$\{\rx(\check\eta_t)\}_{t\in\bbR}$. That is,
\begin{align*}
  \{\hat\eta_{-t}\}_{t\in\bbR}\stackrel{d}{=}\{\rx(\check\eta_t)\}_{t\in\bbR}.
\end{align*}
Indeed, the two processes are exactly equal if $\check\eta_t$ is
constructed using the time-reversal of $\Phi$.  We refrain from
showing this observation in detail, and instead check the local
balance conditions to obtain the product form result.  See
Appendix~\ref{sec:dynReversibility} for definition of local balance
conditions.

\hypertarget{def:Q-i-m}{}\hypertarget{def:Q-i-u}{}

Before we state the main result of this section, we need to fix some
notation.  For any list $\gamma\in O(D,\mbf{C}\times\{\ttu,\ttm\})$
and $i\in\bbN$, we define the $Q^i_\ttu(\gamma)$ to be the number of
unmatched particles \emph{among} the first $i$ particles on $\gamma$, and
define $Q^i_\ttm(\gamma)$ to be the number of matched particles
\emph{excluding} the first $i$ particles of $\gamma$.  In this notation,
we may drop the reference to $\gamma$ when the context is clear.
Also, for the sake of brevity, we will write, for any $n\in\bbN$,
$\rho(n)=2\lambda(D)+n\mu$.  We have the following result.

\begin{theorem}\hypertarget{def:hat-pi}{}
  \label{thm:invMeasure}
  The density of the stationary measure of the backward detailed
  process, $\{\hat\eta_t\}_{t\in\bbR}$, w.r.t.  the measure
  $\oplus_{n=0}^\infty(\lambda\otimes m_c\otimes m_c)^n$ on
  $O(D,\mbf{C}\times\{\ttu,\ttm\})\subset \sqcup_{n=0}^\infty (D\times \mbf{C}\times\{\ttu,\ttm\})^n$, is
  given by
  \begin{align}
    \label{eq:invMeasure}
    \hat\pi(\gamma)&=K\1(\gamma\textrm{ is
                     valid})\prod_{i=0}^{|\gamma|}\frac{1}{\rho(Q^i_\ttu(\gamma))}=:K\1(\gamma\textrm{
                     is valid})\hat\Pi(\gamma),\\
    \hat\pi(\emptyset)&=K,
  \end{align}
  where
  $\hat\Pi(\gamma) =\prod_{i=0}^{|\gamma|}\frac{1}{\rho(Q^i_\ttu(\gamma))}$,
  and where $K$ is the normalizing constant,
  \begin{align*}
    K^{-1}=\sum_{n=0}^\infty\int_{(D\times \mbf{C}\times
    \{\ttu,\ttm\})^n}\1(\gamma \textrm{ is valid})\prod_{i=0}^n\frac{1}{\rho(Q^i_\ttu(\gamma))}(\lambda\otimes
    m_c\otimes m_c)^{(n)}(d\gamma).
  \end{align*}
\end{theorem}
The proof of the above theorem is given in
Appendix~\ref{sec:calculations}.

For the stationary distribution $\pi$ of the original process
$\eta_t$, we compute the marginals of $\hat\pi$.  Firstly, a state
$\gamma=(x_1,\ldots,x_n)\in O(D,\mbf{C})$ is a valid state of the
process, $\{\eta_t\}_{t\in\bbR}$, if and only if
$\{x_1,\ldots, x_n\}\cap N(\gamma)=\emptyset$.

\begin{corollary}\hypertarget{def:pi}{}
  The density of the stationary distribution $\pi$ of the process
  $\eta_t$, w.r.t.  the $\oplus_{n=0}^\infty\blambda^n$ on
  $\sqcup_{n=0}^\infty (D\times\mbf{C})^n$, is given by
  \begin{align}
    \label{eq:non-detailed}
    \pi(\gamma)&=K\1(\gamma\textrm{ is valid})\prod_{i=1}^{|\gamma|}\frac{1}{\blambda(N(\gamma_1^{i}))+i\mu},\\
    \pi(\emptyset)&=K,
  \end{align}
  where $K$ is the same normalizing constant as in
  Theorem~\ref{thm:invMeasure}
\end{corollary}
\begin{proof}
  We calculate the marginal distribution of the unmatched particles
  from the distribution in \ref{thm:invMeasure}.  Given that
  $\eta=(x_1,\ldots,x_n)$ is in the steady state, let $l_i$,
  $1\leq i\leq n$, denote the number of matched particles present
  between $x_i$ and $x_{i+1}$ in the detailed version of the process.
  The only restriction that these particles must satisfy is that they
  must be incompatible with $x_1,\ldots x_i$.  Integrating over
  the positions of each of the $l_i$ particles gives a factor
  $\blambda(D\times\mbf{C}\backslash N(x_1,\ldots, x_i))$. Thus, we
  have:
  \begin{align*}
    \pi(x_1,\ldots,x_n)
    &=K\1(x_1^n\textrm{ is valid})\prod_{i=1}^n\sum_{l_i\in\bbN}\frac{(\blambda(D\times\mbf{C}\backslash
      N(x_1,\ldots, x_i)))^{l_i}}{(2\lambda(D)+i\mu)^{l_i+1}}\\
    &=K\1(x_1^n\textrm{ is valid})\prod_{i=1}^n\frac{1}{2\lambda(D)+i\mu-\blambda(D\times\mbf{C}\backslash
      N(x_1,\ldots,x_i))}\\
    &=K\1(x_1^n\textrm{ is valid})\prod_{i=1}^n\frac{1}{\blambda(N(\eta_1^i))+i\mu}.
  \end{align*}
\end{proof}

\subsection{Clustering Properties and the FKG Property}
\label{sec:clus-prop-fkg}
\hypertarget{def:tilde-pi}{}\hypertarget{def:P(C)}{}

In this section, we focus on the stochastic geometric properties of
the steady state arrangement of the particles in space $D$.  Hence, we forget the
reference to the order of the particles in the steady state.  The
Janossy density (\cite{daley2007introduction}) of a point process
intuitively is the relative probability of observing a given
configuration of points with respect to a given reference measure.
The Janossy density of the steady state distribution of our point
process model, with respect to the Poisson point process on
$D\times\mbf{C}$ with intensity $\blambda$ is given by dropping the
order of particles in Equation~\ref{eq:non-detailed}.  That is, the
Janossy density is
\begin{equation}
  \begin{aligned} \label{eq:janossy}
    \tpi(x_1^n)&=K\1(x_1^n\textrm{ is valid})\tilde{\Pi}(x_1^n),\\
    \tilde\Pi(x_1^n)&=\sum_{(X_1^n)\in \mathcal{P}(x_1^n)}\prod_{i=1}^n\frac{1}{\blambda(N(X_1^i))+i\mu},
  \end{aligned}
\end{equation}
where $\mathcal{P}(x_1^n)$ is the set of all permutations of
$x_1,\ldots,x_n$, and $K$ is a normalizing constant.

Let us take a moment to interpret the term $\blambda(N(x_1^i))$ that
appears in the above expression.  This is the sum of the volumes of
the union of balls around red particles in $x_1,\ldots, x_i$ and the
union of balls around the blue particles in $x_1,\ldots,x_i$.  Since
such terms appear in the denominator in eq.~\ref{eq:janossy}, we
expect that in the steady state the particles of the same color are
clumped together.

In a variety of point processes, such as the one-particle
Widom-Rowlinson model, or certain Cox processes
\cite{blaszczyszyn2014comparison}, the FKG lattice property is a
useful tool for proving stochastic dominance and clustering
properties.  In the case of Widom-Rowlinson model, the FKG inequality
is also useful in showing the existence of a phase-transition
\cite{chayes1995analysis}.  The FKG lattice property defined on a
measure $\psi$ over a finite distributive lattice $\Omega$ states that
for every $\xi,\gamma\in \Omega$,
\begin{align}
  \label{eq:gen-FKG-property}
  \psi(\xi\vee \gamma)\psi(\xi\wedge\gamma)\geq\psi(\xi)\psi(\gamma).
\end{align}
If $\psi$ satisfies eq.~\ref{eq:gen-FKG-property}, it is said to be
log-submodular.  The FKG lattice property implies the positive
association inequality:
\begin{align}
  \psi(fg)\geq \psi(f)\psi(g),\label{eq:positive-assoc}
\end{align}
for all increasing functions $f$ and $g$ on $\Omega$, where $\psi(f)$
represents the expectation of $f$ with respect to $\psi$.

This theorem can also be extended to point processes in the continuum
as follows (see \cite{georgii1997stochastic,chayes1995analysis} for
details).  Let $P$ is point process on a measurable space $S$, with
Janossy density $\psi$ with respect to a Poisson point process with
intensity $\lambda$, the FKG lattice property states that:
\begin{align}\label{eq:cont-FKG-property}
  \psi(\xi\cup \gamma)\psi(\xi\cap\gamma)\geq\psi(\xi)\psi(\gamma),
  \textrm{ forall }\xi,\gamma\in M(S).
\end{align}
Under this hypothesis, one can conclude positive association
inequalities such as eq.~\ref{eq:positive-assoc}, where $f$ and $g$
are now increasing functions on $M(S)$.

\begin{remark}
  The FKG lattice property in the continuum point process case can
  also be stated in terms of the Papangelou conditional intensities:
  If $\varphi(x,\xi)$ is the Papangelou conditional intensity of a
  point process with Janossy density $\psi$, then
  eq.~\ref{eq:cont-FKG-property} is equivalent to
  $\varphi(x,\xi)\geq \varphi(x,\xi')$ for all $x\in S$ and
  $\xi,\xi'\in M(S)$ with $\xi\supseteq\xi'$ (see
  \cite{georgii1997stochastic} for details).
\end{remark}

In the following, we prove an FKG lattice property in the steady state
version of our model, under a specific lattice structure defined on
$M(D,\mbf{C})$.  Let $\xi=(\xi^\ttR,\xi^\ttB)$ and
$\gamma=(\gamma^\ttR,\gamma^\ttB)$ be two configurations in
$M(D,\mbf{C})$, where $\xi^\ttR$ and $\gamma^\ttR$ are the red
particles, and $\xi^\ttB$ and $\gamma^\ttB$ are the blue particles in
these configurations.  We say that $\xi>\gamma$ if and only if
$\xi^\ttR\supset \gamma^\ttR$ and $\xi^\ttB\subset \gamma^\ttB$.  We
note that the FKG lattice property is satisfied in the binary particle
Widom-Rowlinson model with the same lattice structure. In
\cite{chayes1995analysis}, the authors also use discretization based
arguments to lift the positive associations result for this lattice
structure in the continuum.

To prove the FKG lattice property in our setting, we need the
following auxiliary lemma.

\hypertarget{def:paths-space}{}
\begin{lemma}\label{lem:aux}
  Let $(\alpha_i)_{i=1}^n$ and $(\beta_j)_{j=1}^m$ be two sets of
  positive numbers.  Let $P(n,m)$ be the set of all increasing paths
  in the grid $[n]\times [m]$, so that for any $\sigma\in P(n,m)$, we
  have $\sigma(0)=(0,0)$, $\sigma(m+n)=(n,m)$, and
  $\sigma(i+1)-\sigma(i)$ is either $(1,0)$ or $(0,1)$, for all
  $0\leq i<m+n$.  Then, we have
  \begin{align*}
    \sum_{\sigma\in P(n,m)}\prod_{i=1}^{n+m}\frac{1}{\alpha_{\sigma_x(i)}+\beta_{\sigma_y(i)}}=\prod_{i=1}^n\frac{1}{\alpha_{i}}\prod_{i=1}^m\frac{1}{\beta_i},
  \end{align*}
  Here, $\sigma_x$ and $\sigma_y$ denote the $x$ and $y$ coordinate
  respectively.
\end{lemma}
The proof is by induction on $m$.  The proof is not central to the
current discussion, so we present it in Section~\ref{sec:proof-lemma}.

We are now ready to prove a weak form of FKG lattice property for the
Janossy density given in eq.~\ref{eq:janossy}.
\begin{theorem}
  Let $\xi$ and $\gamma$ be disjoint finite subsets of
  $D\times\mbf{C}$ such that the set $\xi\cup \gamma$ is valid.  Then,
  we have
  \begin{align}
    \label{eq:1}
    \tpi(\xi\cup \gamma)\tpi(\emptyset)\geq \tpi(\xi)\tpi(\gamma).
  \end{align}
  Moreover, if $\xi$ and $\gamma$ are such that
  $N(\xi)\cap N(\gamma)=\emptyset$, then equality holds.
\end{theorem}
\begin{proof}\hypertarget{def:order-bijections}{}
  Let $|\xi|=n$ and $|\gamma|=m$.  There is a canonical bijection
  between $\mathcal{P}(\xi\cup \gamma)$ and
  $\mathcal{P}(\xi)\times \mathcal{P}(\gamma)\times P(n,m)$.  For
  $(\mbf{a},\mbf{b},\sigma)\in \mathcal{P}(\xi)\times
  \mathcal{P}(\gamma)\times P(n,m)$ we denote by $(\sigma,\mbf{ab})$
  the corresponding element in $\mathcal{P}(\xi\cup \gamma)$.  We will
  use the coordinate-wise notation $\sigma=(\sigma_x,\sigma_y)$ for
  $\sigma\in\bbZ^2$.  Also, for any sets $A,C\subset D\times\mbf{C}$,
  let $N_{C}(A)=N(A\cap C)$.

  We have
  \begin{equation}
    \begin{aligned}\label{eq:proof1}
      \tilde\Pi(\xi\cup \gamma)
      &=\sum_{\substack{\mbf{a}\in\mathcal{P}(\xi),\mbf{b}\in\mathcal{P}(\gamma)\\\sigma\in P(n,m)}}\prod_{l=1}^{n+m}\frac{1}{\blambda(N((\sigma,\mbf{ab})_1^l))+l\mu}\\
      &\geq
      \sum_{\substack{\mbf{a}\in\mathcal{P}(\xi),\mbf{b}\in\mathcal{P}(\gamma)\\\sigma\in P(n,m)}}\prod_{l=1}^{n+m}\frac{1}{\blambda(N_{\xi}((\sigma,\mbf{ab})_1^l))+\blambda(N_{\gamma}((\sigma,\mbf{ab})_1^l))+l\mu}\\
      &=
      \sum_{\substack{\mbf{a}\in\mathcal{P}(\xi),\mbf{b}\in\mathcal{P}(\gamma)\\\sigma\in
          P(n,m)}}\prod_{l=1}^{n+m}\frac{1}{\blambda(N(\mbf{a}_1^{\sigma_x(l)}))+\blambda(N(\mbf{b}_1^{\sigma_y(l)}))+l\mu}\\
      &=\tilde\Pi(\xi)\tilde\Pi(\gamma),
    \end{aligned}
  \end{equation}
  where in the second step we have used that
  $$\blambda(N((\sigma,\mbf{ab})_1^l))\leq
  \blambda(N_{\xi}((\sigma,\mbf{ab})_1^l))
  +\blambda(N_{\gamma}((\sigma,\mbf{ab})_1^l)),$$ and in the last step
  we used Lemma~\ref{lem:aux} with $\alpha_i=\lambda(N(X_1^i))+i\mu$ and
  $\beta_i=\lambda(N(Y_1^i))+i\mu$.  The proof now follows, since
  $K=\tilde\pi(\emptyset)$.  Note also that if
  $N(\xi)\cap N(\gamma)=\emptyset$, then equality holds in
  eq.~\ref{eq:proof1}.
\end{proof}
The above theorem is useful in the proof of the main result of this
section only through the following corollary.
\begin{corollary}\label{cor:product-cor}
  If $\gamma=\gamma^{\ttR }\cup \gamma^{\ttB}$ is a valid
  configuration (i.e., $\gamma\cap N(\gamma)=\emptyset$), where
  $\gamma^{\ttR }$ and $\gamma^{\ttB}$ are the red and blue particles
  respectively in $\gamma$.  Then,
  \begin{align*}
    \tpi(\gamma) =\frac{1}{K} \tpi (\gamma^{\ttR })
    \tpi(\gamma^{\ttB}),
  \end{align*}
  or equivalently,
  \begin{align*}
    \tPi(\gamma) = \tPi (\gamma^{\ttR })
    \tPi(\gamma^{\ttB}).
  \end{align*}
\end{corollary}

We now state the FKG lattice property for the usual subset ordering
for the same type of particles.
\begin{theorem}\label{thm:fkg-same-type}
  Let $\xi$ and $\gamma$ be finite subsets of $D\times \{\ttR\}$.  Then,
  \begin{align}
    \label{eq:fkg-same-type}
    \tPi(\xi\cup \gamma)\tPi(\xi\cap \gamma)\geq \tPi(\xi)\tPi(\gamma).
  \end{align}
  A similar property holds when $\xi$ and $\gamma$ are finite subsets of
  $D\times \{\ttB\}$.
\end{theorem}
The statement of the previous theorem is combinatorial in nature.
However, its proof is interesting since we were able to use
probabilistic tools by introducing artificial randomness.  In
particular, at one stage in the proof, we employ the FKG inequality on
the lattice $\{0,1\}^n$ (for some $n\in\bbN$).  For the sake of exposition,
this proof is moved to Appendix~\ref{sec:proof-theorem-fkg}.

Using Corollary~\ref{cor:product-cor} and
Theorem~\ref{thm:fkg-same-type}, we conclude the main result of this
section.
\begin{corollary}
  \label{cor:fkg}
  For any two finite subsets, $\xi=(\xi^\ttR,\xi^\ttB)$ and
  $\gamma=(\gamma^\ttR,\gamma^\ttB)$ of $D\times \mbf{C}$, we have
  \begin{align}
    \label{eq:fkg}
    \tPi(\xi\vee\gamma)\tPi(\xi\wedge\gamma)\geq \tPi(\xi)\tPi(\gamma),
  \end{align}
  where the
  $\xi\vee\gamma = (\xi^\ttR\cup\gamma^\ttR,\xi^\ttB\cap \gamma^\ttB)$
  and
  $\xi\wedge\gamma = (\xi^\ttR\cap\gamma^\ttR,\xi^\ttB\cup
  \gamma^\ttB)$.
\end{corollary}
\begin{proof}
  By Corollary~\ref{cor:product-cor} we have
  \begin{align*}
    \tPi(\xi)\tPi(\gamma)=\tPi(\xi^\ttR)\tPi(\xi^\ttB)\tPi(\gamma^\ttR)\tPi(\gamma^\ttB).
  \end{align*}
  By Theorem~\ref{thm:fkg-same-type}, we have
  \begin{align*}
    \tPi(\xi^\ttR)\tPi(\gamma^\ttR)\tPi(\xi^\ttB)\tPi(\gamma^\ttB)\leq \tPi(\xi^\ttR\cup\gamma^\ttR)\tPi(\xi^\ttB\cap\gamma^\ttB)\tPi(\xi^\ttR\cap\gamma^\ttR)\tPi(\xi^\ttR\cup\gamma^\ttB).
  \end{align*}
  Now, since $\xi\vee \gamma$ and $\xi\wedge \gamma$ are valid
  configurations, the proof follows by using
  Corollary~\ref{cor:product-cor} again.
\end{proof}

From Corollary~\ref{cor:fkg} and the FKG inequality (see Appendix in
\cite{chayes1995analysis}), we can conclude that the stationary
measure is positively associated, i.e., for any two increasing
functions $f$ and $g$, then
\begin{align}\label{eq:posi-asso}
  \EE_{\tilde\eta}f(\tilde\eta)g(\tilde\eta)\geq\EE_{\tilde\eta}f(\tilde\eta)\EE_{\tilde\eta}g(\tilde\eta),
\end{align}
where $\tilde\eta$ is a version of the unordered stationary process,
and has the density $\tilde\pi$ with respect to the Poisson point
process with intensity $\blambda$. The above positive association
inequalities also imply that the marginal point processes of the red
(or the blue points) are also positively associated. This can be seen
by taking increasing functions $f$ and $g$ that depend only on the red
points (or the blue points). From this result and Corollary 3.1 of
\cite{blaszczyszyn2014comparison}, we can conclude that
$\tilde\eta^\ttR$ and $\tilde\eta^\ttB$ are weakly-super
Poissonian. Intuitively, this means that the points are more clustered
than the points in a Poisson point process of the same intensity.

\subsection{Boundary Conditions and Monotonicity}
\label{sec:monot-extens-meas}
In this section, we assume that $D$ is a compact subset of the
Euclidean domain $\bbR^d$, for some $d\geq 1$. We will use the FKG
lattice property to prove monotonicity of measures under different
boundary conditions. To state these theorems we will require the
following notation. Let $\zeta\subset \bbR^d\backslash D\times\mbf{C}$
be a valid state, i.e., $N(\zeta)\cap\zeta=\emptyset$. For any such
boundary condition, we define a measure on $M(D,\mbf{C})$ with Janossy
density
\begin{equation}
  \label{eq:2}
  \begin{aligned}
    \tpi_{D,\zeta}(x_1^n)&=K_{D,\zeta}\1(N(x_1^n)\cap (\zeta\cup
    x_1^n)=\emptyset)\tPi_{\zeta}(x_1^n),\\
    \tPi_\zeta(x_1^n)&=\sum_{(X_1^n)\in\cP(x_1^n)}\prod_{i=1}^n\frac{1}{\blambda(N(X_1^i)\cap
      N(\zeta)^{c})+i\mu}.
  \end{aligned}
\end{equation}
Three important boundary conditions are
$\zeta = (\bbR^d\backslash D)\times\{\ttR\}$, $\zeta = \emptyset$ and
$\zeta = (\bbR^d\backslash D)\times\{\ttB\}$. These are termed the
\emph{red}, the \emph{free} and the \emph{blue} boundary conditions
respectively. We use special notation for the densities with these
boundary conditions, namely, $\tpi_{S,\ttR}$, $\tpi_{S}$ and
$\tpi_{S,\ttB}$.

The boundary conditions can also be partially ordered: let
$\zeta_1\geq \zeta_2$ if and only if $\zeta_1^\ttR\supset\zeta_2^\ttR$
and $\zeta_1^\ttB\subset\zeta_2^\ttB$. We are now in a position to
state the first result of this section.

\begin{theorem}
  Let $D\subset \bbR^d$ be compact set. Let $\zeta_1\geq \zeta_2$ be
  two boundary conditions on $D$. Then, the measure with density
  $\tpi_{D,\zeta_1}$ stochastically dominates the measure with density
  $\tpi_{D,\zeta_2}$.
\end{theorem}
\begin{proof}[Outline of the proof]
  By Holley's inequality \cite{holley1974remarks}, it is enough to
  prove that for two states, $\eta$ and $\gamma\in M(D,\mbf{C})$, we
  have
  \begin{align}
   \tpi_{D,\zeta_1}(\eta\vee\gamma) \tpi_{D,\zeta_2}(\eta\wedge\gamma)\geq \tpi_{D,\zeta_1}(\eta)\tpi_{D,\zeta_2}(\gamma).
  \end{align}
  We first note that if $N(\eta)\cap( \zeta\cup\eta)=\emptyset$ and
  $N(\gamma)\cap(\zeta_2\cup\gamma)=\emptyset$, then
  \begin{align*}
    N(\eta\vee\gamma)\cap(\zeta_1\cup(\eta\vee\gamma))=N(\eta\wedge\gamma)\cap(\zeta_2\cup((\eta\wedge\gamma))=\emptyset.
  \end{align*}
  Since the red and blue subsets do not interact by
  Corollary~\ref{cor:product-cor}, we only need to show that if
  $\zeta_2\subset\zeta_1\subset(\bbR^d\backslash D)\times\{\ttR\}$ and
  $\eta$ and $\gamma\in M(D,\{\ttR\})$, then
  \begin{align*}
   \tPi_{\zeta_1}(\eta\cup\gamma)\tPi_{\zeta_2}(\eta\cap\gamma)\geq\tPi_{\zeta_1}(\eta)\tPi_{\zeta_2}(\gamma).
  \end{align*}
  The proof of the last statement follows by simple modification of
  the proof of Theorem~\ref{thm:fkg-same-type}, presented in
  Appendix~\ref{sec:proof-theorem-fkg}. Specifically,
  Equation~\ref{eq:fkg-4} in the appendix is modified to
  \begin{equation}
\begin{aligned}\label{eq:fkg-4-new}
  &
  \EE\prod_{i=1}^{n+m+2k}\left(\substack{\blambda(N(\mathfrak{a}_{1}^{\mathfrak{S}_x(i)})\cap
        N(\zeta_1)^c)+\blambda(N(\mathfrak{b}_{1}^{\mathfrak{S}_y(i)})\cap N(\zeta_2)^c)+\blambda(N(\mathfrak{c}_{1}^{\mathfrak{S}_z(i)})\cap N(\zeta_1)^c)+\blambda(N(\bar{\mathfrak{c}}_{1}^{\mathfrak{S}_w(i)})\cap N(\zeta_2)^c)\\-\blambda(N(\mathfrak{a}_{1}^{\mathfrak{S}_x(i)})\cap
  N(\mathfrak{c}_{1}^{\mathfrak{S}_z(i)})\cap N(\zeta_1)^c)-\blambda(N(\mathfrak{b}_{1}^{\mathfrak{S}_y(i)})\cap
  N(\mathfrak{c}_{1}^{\mathfrak{S}_z(i)})\cap N(\zeta_2)^c)+i\mu}\right)^{-1}\\
  &\geq\EE\prod_{i=1}^{n+m+2k}\left(\substack{\blambda(N(\mathfrak{a}_{1}^{\mathfrak{S}_x(i)})\cap N(\zeta_1)^c)+\blambda(N(\mathfrak{b}_{1}^{\mathfrak{S}_y(i)})\cap N(\zeta_2)^c)+\blambda(N(\mathfrak{c}_{1}^{\mathfrak{S}_z(i)})\cap N(\zeta_1)^c)+\blambda(N(\bar{\mathfrak{c}}_{1}^{\mathfrak{S}_w(i)})\cap N(\zeta_2)^c)\\-\blambda(N(\mathfrak{a}_{1}^{\mathfrak{S}_x(i)})\cap
  N(\mathfrak{c}_{1}^{\mathfrak{S}_z(i)})\cap N(\zeta_1)^c)-\blambda(N(\mathfrak{b}_{1}^{\mathfrak{S}_y(i)})\cap
  N(\bar{\mathfrak{c}}_{1}^{\mathfrak{S}_w(i)})\cap N(\zeta_2)^c)+i\mu}\right)^{-1},
\end{aligned}
\end{equation}
where the expectation is over a uniformly random
choice
\begin{align*}
  &(\mathfrak{S,abc\bar{c}})\in
  P(n,m,k,k)\times\cP(\eta\backslash\gamma)\times\cP(\gamma\backslash\eta)\times\cP(\eta\cap\gamma)\times\cP(\eta\cap\gamma),\\
  &n=|\eta\backslash\gamma|,\ m=|\gamma\backslash\eta|,\
  k=|\eta\cap\gamma|.
\end{align*}
The rest of the proof follows similar steps to the proof in
Appendix~\ref{sec:proof-theorem-fkg}.
\end{proof}

\begin{remark}
  The above proof can be suitably modified to give the following
  interesting result.  Let $D_1\subset D_2$ be two compact subsets of
  $\bbR^d$. With an abuse of notation, let $\tpi_{D_2,\ttR}(\eta)$
  denote the marginal-Janossy density of observing $\eta$ in $D_1$,
  under the measure with density $\tpi_{D_2,\ttR}$. Similarly, we
  overload the notation for $\tpi_{D_2,\ttB}$. With this notation, we
  may prove that $\tpi_{D_1,\ttR}\geq\tpi_{D_2,\ttR}$ and
  $\tpi_{D_1,\ttB}\leq\tpi_{D_2,\ttB}$. For the proof, we apply
  Holley's inequality, which requires that the following inequality
  holds:
  \begin{align*}
    \tpi_{D_1,\ttR}(\eta\vee\gamma)\tpi_{D_2,\ttR}(\eta\wedge\gamma)\geq\tpi_{D_1,\ttR}(\eta)\tpi_{D_2,\ttR}(\gamma),
  \end{align*}
  where $\eta,\gamma\in M(D_1,\mbf{C})$ are two valid
  configurations. This is follows from the inequality
  \begin{align*}
    \tpi_{D_1,\ttR}(\eta\vee\gamma)\tpi_{D_2,\ttR}(\xi\cup(\eta\wedge\gamma))\geq\tpi_{D_1,\ttR}(\eta)\tpi_{D_2,\ttR}(\xi\cup\gamma),
  \end{align*}
  where $\xi\in M(D_1\backslash D_2,\mbf{C})$ is any configuration
  such that $\xi\cup\gamma$ is a valid configuration. The later
  inequality can be proved using similar ideas used in the proof of
  Theorem \ref{thm:fkg-same-type} in Appendix
  \ref{sec:proof-theorem-fkg}. We skip the details of the cumbersome
  calculations here. We will only remark here that, such a
  monotonicity of measures allows us to consider the limiting extremal
  measures $\lim_{D_n\nearrow\bbR^d}\tpi_{D_n,\ttR}$ and
  $\lim_{D_n\nearrow\bbR^d}\tpi_{D_n,\ttB}$, on the infinite Euclidean
  domain $\bbR^d$. In the next few sections, we will consider the
  First-in-first-match process on infinite Euclidean domains, and
  prove the existence of a stationary regime. We leave the job of
  exploring of the connection between the stationary measure so
  obtained and these limiting measures to future work.
\end{remark}

\section{First-in-First-Match Matching Process on Euclidean domains}
\label{sec:defi-eucl}
In the following few sections, we extend the definition of the process
previously given on a compact space to a non-compact space. We will
specifically focus on the Euclidean space $\bbR^d$, for some
$d\geq 1$.  The following methodology can be extended to other
non-compact spaces that satisfy certain additional assumptions, but we
refrain from presenting these results in complete generality.  When
$D=\bbR^d$, there are infinitely many arrival and departure events
that are triggered in any finite interval of time.  So, the process
cannot be constructed as a jump Markov process using the algorithm
presented in Section~\ref{sec:form-desc}.

\hypertarget{def:kappa}{}

The key to the definition and construction of the process on $\bbR^d$
is the following viewpoint.  Let us look understand this viewpoint the
compact setting first, and see its relation to the algorithm given in
Section~\ref{sec:form-desc}.  Let $D$ be a compact space.  Let
$\Phi\in M(D\times\bbR^+,\mbf{C}\times\bbR^+)$ and
$\eta_0\in O(D,\mbf{C}\times \bbR^+)$ be the driving Poisson point
process and the initial condition, respectively, as defined in the
Section~\ref{sec:form-desc}.  Here, each particle $x\in \eta_0$ is of
the form $x=(p_x,c_x,w_x)$, where $p_x$, $c_x$ and $w_x$ are the
position, color and patience of the particle.  Similarly, any point
$x\in \Phi$ is of the form $x=(p_x,b_x,c_x,w_x)$, where additionally
$b_x$ denotes the arrival time in $\Phi$.

We will treat $\Phi$ as an element of
$O(D\times\bbR^+,\mbf{C}\times\bbR^+)$, where the points of $\Phi$ are
ordered according to their birth times, as in
Section~\ref{sec:form-desc}. We will also treat $\eta_0$ as an element
of $O(D\times\bbR^+,\mbf{C}\times\bbR^+)$, by setting $b_x=0$ for all
$x\in\eta_0$, while preserving the order in $\eta_0$.  Moreover, we
also consider $\Phi\cup\eta_0$ as an element of
$O(D\times\bbR^+,\mbf{C}\times\bbR^+)$, where all the elements of
$\eta_0$ are ranked less than the elements of $\Phi$, while preserving
the order within these sets.

We define a function
$\kappa:\Phi\cup\eta_0\ra D\times\mbf{C}\times\bbR^+\sqcup
\{\diamondsuit\}$, which we call the \emph{killing} function, that is
created as the process is built by the algorithm in
Section~\ref{sec:form-desc}.  We set $\kappa(x)$ according to the
following exhaustive set of rules.
\begin{enumerate}
\item If $x$ arrives after $y$ and matches with it, then
  $\kappa(x)=y$.
\item If $x$ is accepted into the system, then
  $\kappa(x)=\diamondsuit$.
\item If $x\in\eta_0$, then $\kappa(x)=\diamondsuit$.
\end{enumerate}
According to the description of the process, the function $\kappa$
satisfies the following recursive property.  For any $x\in \Phi$,
\begin{align}\label{eq:kappa}
  \kappa(x)=\min \left\{y\in (\Phi\cup \eta_0):
  \substack{
  y<x,\ c_y\neq c_x,\ d(p_y,p_x)<1,\  \kappa(y)=\diamondsuit,\ b_y+w_y>b_x,\\
  (\forall z,\ y<z<x,\ d(p_y,p_z)<1,\ c_z=c_x,\ \kappa(z)\neq y)
  }\right\},
\end{align}
where the minimum above is set equal to $\diamondsuit$ if the above
set is empty.  In words, the conditions in the definition of the above
set select particles of opposite color that arrive before (or are
ranked lower), are accepted when they arrive, whose patience does not
run out before $x$ arrives, and are not matched to any particle
arriving before $x$.

The above recursive property can serve as a definition of the function
$\kappa$, even in the non-compact case, if we can show that the
recursive definition terminates almost surely.  We note that we could
compute the value of $\kappa(x)$ if we knew all values of $\kappa$ on
points in $\Phi\cup \eta_0$ that are within a spatial distance $2$
from $x$ and that arrive before it.  It is also enough to just know
all the values of $\kappa$ for point in $\Phi$ within a spatial
distance $4$ that arrive before $x$.  The following lemma provides the
tool needed to claim the termination of the recursive definition.

\begin{lemma}
  Suppose $\Phi\in O(\bbR^d\times \bbR^+,\mbf{C}\times\bbR^+)$ be a
  Poisson point process of intensity
  $\ell\otimes\ell\otimes m_c\otimes \mu e^{-\mu x}\ell(dx)$, where
  $\ell$ is the Lebesgue measure on corresponding Euclidean spaces.
  Then, there is no infinite sequence
  $\{y_{n}\}_{n\in\bbN} \subset \Phi$ such that $b_{y_i}>b_{y_{i+1}}$
  and $d(p_{y_i},p_{y_{i+1}})\leq 4$ for all $i>0$.
\end{lemma}
\begin{proof}
  Let $\epsilon>0$ be fixed.  Consider the event
  $T_\epsilon:=\cup_{q\in\bbQ^+}T_{\epsilon,q}$, where
  $T_{\epsilon,q}$, $q\in\bbQ^{+}$, is the event that the Random
  Geometric graph (see \cite{meester1996continuum}) obtained by
  joining any two points in $\{y\in \Phi: b\leq b_y\leq b+\epsilon\}$
  whose positions are within a distance $4$ from each other does not
  percolate.  From Theorem 3.2 of \cite{meester1996continuum}, we know
  that $\PP(T_{\epsilon,b})=1$ for every $b\in \bbQ^+$ if $\epsilon$
  is small enough.  Hence, $\PP(T_\epsilon)=1$ for small enough
  $\epsilon$.  This precludes the presence of an infinite sequence
  $y_1,y_2\ldots\in \Phi$ such that $b_{y_i}>b_{y_{i+1}}$ and
  $d(p_{y_{i}},p_{y_{i+1}})\leq 4$, for all $i\in\bbN$.  Indeed, if
  such a sequence exists then the limit $b=\lim_{i\ra\infty}b_{y_i}$
  exists and by density of $\bbQ^{+}$ in $\bbR^{+}$, this event
  belongs to the set $T_{\epsilon}^{c}$.
\end{proof}

The above lemma ensures that we can obtain the value of $\kappa(x)$
for any $x\in\Phi$ by recursively applying the
Equation~\ref{eq:kappa}.  The process in turn can be defined by
setting
\begin{align*}
  \eta_t=\{(p_y,c_y):y\in \Phi\cup\eta_0,\ \kappa(y)=\diamondsuit,\ b_y\leq t< b_y+w_y,
  \textrm{ and } \kappa(z)=y\implies t<b_{z}\},
\end{align*}
for all $t>0$. On the unbounded domain $\bbR^d$, this will serve as
the definition of the FIFM spatial matching process.

\subsection{Construction of Stationary Regime on  Euclidean domains}
\label{sec:stat-eucl}
The simple coupling form the past argument presented in
Section~\ref{sec:form-desc} does not hold in the case where
$\lambda(D)=\infty$, since we cannot find a sequence of regeneration
times going to $\-\infty$ in this case.  We can show however that a
coupling from the past argument can still be performed locally in
space.  This is done by first showing that, for a compact subset $C$
of the domain $\bbR^d$, there exists time $T_C$ beyond which two
simulations agree for all times $t$ beyond time $T_C$ (So, $T_C$ is
not a stopping time).  A key ingredient in proving the existence of
$T_C$ is an analysis of the decay in first order moment measures of
the discrepancies between the two point patterns in simulation.  Using
this analysis, we are also able to bound the moments of $T_C$, which
enables the application of ergodic theorems, as applied in the simple
coupling from the past construction in Section~\ref{sec:form-desc}.
In the following sections, we present this coupling from the past
argument in detail.

\subsection{A Coupling of Two Processes}
\label{sec:couplingTwoProcs}{}\hypertarget{def:specials}{}
We will first obtain some results about a coupling of two different
processes, $\{\eta_t^1\}_{t\geq 0}$ and $\{\eta_t^2\}_{t\geq 0}$,
starting from the two different initial conditions, but driven by the
same driving process
$\Phi\in O(\bbR^d\times\bbR,\mbf{C}\times\bbR^+)$.  Let $\eta^1_0$ and
$\eta^2_0$ be the two valid initial conditions
($\eta^i_0\cap N(\eta^i_0)=\emptyset$).  At any time $t\geq 0$, there
are some particles that are present in both processes.  These
particles will be called \emph{Regular} particles, and denoted by$R_t$.
Call those particles that are present in $\eta^1_t$ and absent in
$\eta^2_t$ as \emph{Zombies}, and those that are absent in $\eta^1_t$
and present in $\eta^2_t$ as \emph{Antizombies}.  We denote them by
$Z_t$ and $A_t$ respectively.  Further, call particles in
$Z_t\cup A_t$ as \emph{Special} particles, and denote them by $S_t$.

We now prove that the density of the special particles
decays exponentially to zero.

\begin{theorem}\label{thm:exponentialDecay}
  There exist constants $c>0$, $\beta_{S_t}<\beta_{S_0} e^{-ct}$, for
  all $t>0$, where $\beta_{S_t}$ is the intensity of the special
  points, $S_t$.
\end{theorem}
\begin{proof}
  Let $K\subset \bbR^d$ be compact.  Define $K^+=\{y\in \bbR^d:d(y,K)\leq 1\}$
  and $K^-=\{y\in \bbR^d: d(y,K^c)\leq 1\}^c$, $\partial K^+=K^+-K$
  and $\partial K^-=K-K^-$.  Also, let for any $T\subset \bbR^d$,
  $T_{\mbf{C}}$ denote the set $T\times \mbf{C}$.  Now, we will
  compute the difference
  $\EE [Z_{t+\delta}(K_{\mbf{C}})-Z_{t}(K_{\mbf{C}})]$ for small
  $\delta>0$, by tracking the changes that may occur in the short time
  interval $(t,t+\delta)$.  Recall that we use the notation $W^i_x$ to
  denote the domain of influence of the particle $x\in\eta^i_t$, $i=1,2$.
  Also, in the following we have $\blambda:=\ell\otimes m_c$ on
  $\bbR^d\times\mbf{C}$.  The following possibilities may occur:
  \begin{itemize}
  \item A zombie in $ K$ exits on its own by losing patience.  The
    expected difference is
    \begin{align}\label{eq:list1}
      -\mu\delta\EE Z_{t}(K_{\mbf{C}})+o(\delta).
    \end{align}
  \item With probability $o(\delta)$, two or more particles arrive or
    depart in $K^+$.  The expected change in $Z_t(K)$, given that this
    occurs, is $o(\delta)$.
  \item A zombie in $K$ matches with a particle arriving in $K^c$,
    which is accepted in the process $\eta_t^2$.  This results in a
    difference of
    \begin{align}\label{eq:list2}
      -\delta\EE\sum_{x\in Z_t\cap K_{\mbf{C}}}\blambda(W^1_x\cap K^c_{\mbf{C}}\cap (N(\eta_t^2))^c)+o(\delta).
    \end{align}
  \item A zombie in $K$ matches with a particle arriving in $K$, which is
    accepted in the process $\eta_t^2$.  This new particle is an
    antizombie.  The resulting change is
    \begin{align}\label{eq:list3}
      -\delta\EE\sum_{x\in Z_t\cap K_{\mbf{C}}}\blambda(W_x^1\cap K_{\mbf{C}}\cap (N(\eta_t^{2}))^c)+o(\delta).
    \end{align}
  \item A zombie in $K$ matches with an arriving particle, which also
    matches with some particle in $K^c$ in the complementary process.
    This results in an expected change of
    \begin{align}\label{eq:list4}
      -\delta\EE\sum_{x\in Z_t\cap K_{\mbf{C}}}\sum_{y\in\eta_t^{2}\cap K^c_{\mbf{C}}}\blambda(W^1_x\cap W^2_y)+o(\delta).
    \end{align}
  \item A zombie in $K$ matches with an arriving particle, which also
    matches with some anti-zombie in $K$.  This results in an expected
    change of
    \begin{align}\label{eq:list5}
      -\delta\EE\sum_{x\in Z_t\cap K_{\mbf{C}}}\sum_{y\in A_t\cap
      K_{\mbf{C}}}\blambda(W^1_x\cap W_y^2)+o(\delta).
    \end{align}
  \item An anti-zombie matches with a particle arriving in $K$, that is
    accepted in the complementary process.  This particle becomes a
    zombie.  This results in an expected change of
    \begin{align}\label{eq:list6}
      \delta\EE\sum_{x\in A_t}\blambda(W^2_x\cap K_{\mbf{C}}\cap(N(\eta_t^1))^c)+o(\delta).
    \end{align}
  \item An arriving particle matches with a zombie in $K^c$ and a regular
    particle in $\eta_t^2\cap K_{\mbf{C}}$.  The regular particle turns into
    a zombie.  This results in a change of
    \begin{align}\label{eq:list7}
      \delta\EE\sum_{\substack{x\in
      R_t\cap K_{\mbf{C}}\\y\in Z_t\cap K^c_{\mbf{C}}}}\blambda(W^2_x\cap W_y^1)+o(\delta)
    \end{align}
  \end{itemize}
  Hence, we have:
  \begin{align*}
    & \EE[Z_{t+\delta}(K_{\mbf{C}}))-Z_{t}(K_{\mbf{C}})]\\
    &=-\mu\delta\EE Z_t(K_{\mbf{C}}) +\delta\EE\left[-\sum_{x\in
      Z_t\cap K_{\mbf{C}}}\bigg(\blambda(W^1_x\cap
      K^c_{\mbf{C}}\cap(N(\eta_t^2)))\right.\\
    &\left.\IndII-\blambda(W_x^1\cap K_{\mbf{C}}\cap (N(\eta_t^2))^c -\sum_{y\in\eta_t^{2}\cap K^c_{\mbf{C}}}\blambda(W^1_x\cap W^2_y)-\sum_{y\in A_t\cap K_{\mbf{C}}}\blambda(W^1_x\cap W_y^2)\bigg)\right.\\
    &\IndI\left.+\sum_{x\in A_t}\blambda(W^2_x\cap
      K_{\mbf{C}}\cap(N(\eta_t^1))^c)+\sum_{\substack{x\in
      R_t\cap K_{\mbf{C}}\\y\in Z_t\cap K^c_{\mbf{C}}}}\lambda(W^2_x\cap
    W_y^1)\right]+o(\delta).
  \end{align*}
  Taking only the 2nd, 5th and 6th terms of the square braces of the
  above expression, dividing by $\delta$, and taking the limit as
  $\delta\ra 0$, we obtain:
  \begin{align*}
    \frac{d\EE Z_t(K_{\mbf{C}})}{dt}\leq
    &-\mu\EE Z_t(K_{\mbf{C}})+\EE\left[-\sum_{x\in Z_t\cap K_{\mbf{C}}}-\lambda(W_x^1\cap K_{\mbf{C}}\cap (N(\eta_t^2))^c\right.\\
    &\IndI\left.+\sum_{x\in A_t}\lambda(W^2_x\cap K_{\mbf{C}}\cap(N(\eta_t^1))^c)+\sum_{\substack{x\in
      R_t\cap T_{\mbf{C}}\\y\in Z_t\cap T^c_{\mbf{C}}}}\lambda(W^2_x\cap W_y^1)\right].
  \end{align*}
  We have similar bounds for derivatives of $\EE{A_t(K_{\mbf{C}})}$.
  Adding these expressions we obtain:
  \begin{align}
    \frac{d}{dt}\EE{S_t(K_{\mbf{C}})}
    &\leq-\mu\EE{S_t(K_{\mbf{C}})}+\EE\left[2\sum_{\substack{x\in
      R_t\cap K_{\mbf{C}}\\y\in Z_t\cap K^c_{\mbf{C}}}}\blambda(W^2_x\cap
    W_y^1)+2\sum_{\substack{x\in R_t\cap K_{\mbf{C}}\\y\in A_t\cap K^c_{\mbf{C}}}}\blambda(W^2_x\cap W_y^1)\right]\nonumber\\
    &\leq -\mu\EE S_t(K_{\mbf{C}})+4\blambda(\partial K_{\mbf{C}}^+\cup\partial K^-_{\mbf{C}})\nonumber.
  \end{align}
  Taking the limit $K\nearrow\bbR^d$, by spatial ergodicity of the
  process $S_t$, we obtain
  \begin{align}
    \frac{d}{dt}\beta_{S_t}\leq -\mu\beta_{S_t},
  \end{align}
  from which we can conclude that
  $\beta_{S_t}\leq \beta_{S_0}e^{-\mu t}$.
\end{proof}

\subsection{The Coupling from the Past Construction}
\label{sec:CFTP}
In this section, we present the coupling from the past construction of
the stationary regime.  Let $\Phi$ be a doubly infinite Poisson point
process, as in Lemma~\ref{lem:boundedRegeneration}.  That is, $\Phi$ is a Poisson point process defined
on $\bbR^d\times \bbR$, with i.i.d.  marks in $\mbf{C}\times\bbR^+$.
The intensity of the point process is $2\ell\otimes \ell$ and the
marks are independent of each other, with the color uniformly
distributed in $\mbf{C}$ and the patience an exponential random
variable.  Let $\{\theta_t\}_{t\in\bbR}$ be a sequence of time-shift
operators such that
$ \Phi\circ \theta_t(L\times A)=\Phi(L\times (A-t))$.  Let
$\{\eta^T_{t}\}_{t\geq -T}$, $T\in\bbN$, be a sequence of processes
starting at time $-T$ with empty initial conditions and driven by
arrivals from $\Phi$.  We have
$\eta^T_t=\eta^0_{t+T}\circ\theta_{-T}$.

The processes $\eta^1_t$ and $\eta^0_t$ are driven by the same Poisson
point process $\Phi$ beyond time $0$.  Treating the particles in
$\eta^1_0$ as the initial conditions, we have a coupling of
$\{\eta^1_t\}_{t\geq 0}$ and $\{\eta^0_t\}_{t\geq 0}$ as in
Section~\ref{sec:couplingTwoProcs}.  The discrepancies on any bounded
set goes to zero exponentially fast by
Theorem~\ref{thm:exponentialDecay}.  In the following lemma, we show
that such an exponential rate of convergence is enough to show that
the time after which discrepancies never appear in any compact region
has finite expectation.  For any compact $K\subset D$, define
\begin{align}
  \tau^0(K):=\inf\{t>0:\eta_{s}^1|_K=\eta^0_s|_K,\ s\geq t\},
\end{align}
and in the following, let $S_t$ denote the set of discrepancies,
$\eta^0_t\triangle\eta^1_t$.  Note that $\tau^0(K)$ is not a stopping
time in our setting, since, first, $S_t\neq \emptyset$ for all
$t\geq 0$ a.s.  (there are always discrepancies somewhere in $\bbR^d$)
by spatial ergodicity, and second, once discrepancies vanish in $K$,
they can reappear due to interactions with the particles from outside
of $K$.

We have
\begin{lemma}\label{lem:coupling-time-bound}
  For all compact $K\subset\bbR^d$, $\EE\tau^0(K)<\infty$.
\end{lemma}
\begin{proof}
  We view $S_t(K)$, $t\geq 0$, as a birth-death process.  Let
  $S_t(K)=S_0(K) +S^+(0,t] -S^-(0,t]$, where $S^+$ and $S^-$ are
  simple counting processes.  Since new special particles only result
  from interaction of arriving particles with existing special particles,
  the rate of increase in $S^+$ is bounded above by
  $\sum_{x\in S_t\cap K}\ell(B(x,1))=\ell(B(0,1))S_t(K)$.  Hence,
  \begin{align*}
    \EE S^+[0,\infty)&=\leq\ell(B(0,1))\int_0^\infty\EE S_t(K)dt <\infty.
  \end{align*}
  Since total departures are less than total arrivals,
  $\EE S^-[0,\infty)\leq \EE S_0(K)+ \EE S^+[0,\infty)$.  This in
  particular shows that $S^+[0,\infty)$ and $S^-[0,\infty)$ exist and
  are finite a.s.  Thus, $\lim_{t\ra \infty}S_t(K)$ also exists and is
  finite a.s.  By dominated convergence theorem,
  $\lim_{t\ra\infty}\EE S_t(K)=\EE\lim_{t\ra\infty}S_t(K)$.  Thus, by
  Theorem~\ref{thm:exponentialDecay}, $\lim_{t\ra\infty}S_t(K)=0$, a.s.
  This shows that $\tau(K)<\infty$ a.s.

  Further, we have:
  \begin{align*}
    \EE \tau(K)&\leq\EE\int_0^\infty tS^-(dt)\\
               &=\EE\int_0^\infty tS^+(dt)-\EE\int_0^\infty tS(dt)\\
               &=\EE\int_0^\infty tS^+(dt)+E\int_0^\infty S(t)dt\\
               &\leq \ell(B(0,1))\int_0^\infty
                 t\EE S_t(K)dt+\int_0^\infty\EE S_t(K)dt\\
               &<\infty.
  \end{align*}
\end{proof}
Now, let $\tau^T(K)$ be defined as
\begin{align*}
  \tau^T(K):=\inf\{t>-T:\eta_{s}^{T+1}|_K=\eta^T_s|_K,\ s\geq t\}.
\end{align*}
$\tau^T(K)$ denotes the time at which executions of processes
$\{\eta^T_t\}$ and $\{\eta^{T+1}_t\}$ coincide inside the set $K$.  We
have
\begin{align*}
  \tau^T(K)=\tau^0(K)\circ\theta_{-T}-T.
\end{align*}
That is,
\begin{align} \label{eq:ergodicVty}
  \tau^T(K)+T=\tau^0(K)\circ\theta_{-T}.
\end{align}
Therefore, the sequence $\tau^T(K)+T$ is a stationary and ergodic
sequence.  By Birkhoff's point-wise ergodic theorem and by
Lemma~\ref{lem:coupling-time-bound},
\begin{align*}
  \lim_{T\ra\infty}\sum_{i=0}^T\frac{\tau^i(K)+i}{T}&=\EE\tau^0(K)<\infty,\
                                                      \mathrm{a.s.}
\end{align*}
Therefore the last term in the summation, $\frac{V^T_y+T}{T}$ goes to
$0$ as $T\ra\infty$.  From this we conclude that
\begin{align}
  \lim_{T\ra\infty}V^T_y=-\infty.
\end{align}

This result has the following implication.  For every realization of
$\Phi$, any compact set $K$ and $t\in \bbR$, there exists a $k\in\bbN$
such that for all $T>k$, $\tau^0(K)\circ \theta_{-T}-T<t$.  That is,
the execution of all processes $\{\eta^T_s\}$, $T>k$, coincides at
time $t$ on the compact set $K$.  Then, locally in the total variation
sense, the following limit is well-defined a.s.  on the same
probability space:
\begin{align}
  \label{eq:limitingprocess}
  \eta_t:=\lim_{T\ra\infty}\eta^T_t.
\end{align}

The process $\eta$ is $\{\theta_n\}_{n\in\bbZ}$ compatible since
\begin{align*}
  \eta_t\circ\theta_{1}&=\lim_{T\ra\infty}\eta^T_t\circ \theta_1\\
                       &=\lim_{T\ra\infty}\eta^0_{T+t}\circ\theta_{-T+1}\\
                       &=\lim_{T\ra\infty}\eta^0_{t+1+T-1}\circ\theta_{-T+1}\\
                       &=\eta_{t+1}.
\end{align*}
Further, the process can also be shown to be $\{\theta_s\}_{s\in\bbR}$
compatible.  Indeed, fix $s\in \bbR$.  Let us implement a similar
coupling from the past procedure, but with processes $\eta^{T+s}$ that
start with empty initial conditions at time $-T-s$, $T\in \bbN$.  If
$\hat\eta$ is the process obtained in such as manner, it can be shown
that $\{\hat\eta_t\}_{t\in\bbR}$ is equal to $\{\eta_t\}_{t\in\bbR}$
generated as above.  Thus, $\eta_{t+s}=\eta_{t}\circ\theta_s$.  This
proves the $\{\eta_t\}_{t\in \bbR}$ is the stationary regime of the
process.

\section{Conclusion and Future Work}
\label{sec:concluding-remarks}
In this paper, we focused on a dynamic matching model with a natural
policy, under the added assumption that particles may depart without
being matched. We were able to find a characterization of the steady
state distribution of the particles. Then using this characterization,
we proved the FKG lattice property, which in turn enabled us to
conclude that the property that particles of the same type are
weakly-super Poissonian. We also prove that there is a stationary
regime for the dynamics is the infinite Euclidean domain, $\bbR^d$.

The two particle Widom-Rowlinson model is a simpler model, where the
FKG property, as satisfied by our model, is also satisfied. There,
this property is used to show the existence of Markov random fields on
the infinite Euclidean domain, $\bbR^d$. In the future, we would like
to see whether this construction works in our setting, and how it
relates to the stationary regime constructed on the domain $\bbR^d$.

The gray version of two particle WR model, which is obtained by
removing the reference to the colors of the points also satisfies an
FKG inequality -- we are unable to prove this in our setting. This is
a fundamental step in the symmetry breaking argument of
\cite{chayes1995analysis}. We have not found such an argument in our
setting. A symmetry breaking argument will show, for certain values
of the parameter, that there are more red points than blue points in
the steady state, or vice-versa. This also has implications on the
relaxation times of the Markov process on finite domains. In the
future, we would like to explore these problems.

\Authornote{Mayank}{Need to show that the point process is a Markov
  random field. Need to construct the infinite version. Show that
  there is no percolation in the gray version.}

\appendix

\section{Dynamic reversibility of Markov processes}
\label{sec:dynReversibility}
In this section we give a brief discussion of a result needed to
construct the product form distribution.  This concept will be termed
dynamic reversibility of a Markov process, following the terminology in
\cite{kelly2011reversibility}, where the concept was discussed for
Markov processes on countable state spaces.  We thus define it on
countable state spaces first, and then on general state spaces.

Let $\{X(t)\}_{t\in\bbR}$ be a stationary, irreducible continuous-time
Markov process with values in a countable state space $S$.  Let
$q(j,k)$ denote the transition rate from state $j\in S$ to $k\in S$
and let $\pi$ be the stationary distribution of the process.  In this
case, the balance equations are $\sum_{j\in S}\pi(j)q(j,k)=0$.

The reversed process, $X(-t)$, is also a stationary Markov process
with transition rates $q'(j,k)=\frac{\pi(k)q(k,j)}{\pi(j)}$.  The
converse of this statement can be used as a characterization of the
stationary distribution.  We state this result in the following
theorem.
\begin{theorem}\label{thm:guessReversed}
  Let $X(t)$ be a stationary irreducible Markov process with
  transition rates $q(j,k),\ j,k\in S$.  If there exists a collection
  of numbers $q'(j,k),\ j,k\in S$, and a probability measure $\pi$ on
  $S$ such that
  \begin{align}
    \pi(j)q(j,k)=\pi(k)q'(k,j),\quad j,k\in S, \label{eq:local-balance-1}
  \end{align}
  then $\pi$ is the stationary distribution of the process and $q'$ is
  the transition rate matrix for the reversed process.
\end{theorem}
Thus, if we can guess the transition rates of the reversed process and
a stationary measure, we can verify them by checking a local balance
condition of the form eq.~\ref{eq:local-balance-1}.  See Theorem 1.13
of \cite{kelly2011reversibility} for a proof of this result.  In
practice, finding $q'$ is usually as intractable as finding the
stationary distribution directly.  However, occasionally we may come
across pairs of Markov processes that are reversed versions of each
other, perhaps after a transformation of the state space.  We state
this phenomenon in the next theorem.
\begin{theorem}
  \label{thm:dynReversedDiscrete}
  Let $S,T$ be two countable spaces.  Let $X(t)$ and $Y(t)$ be two
  stationary irreducible Markov processes with values in $S$ and $T$,
  and transition matrices $q$ and $q'$ respectively.  Suppose there is
  an isomorphism $\phi:S\ra T$ between the two spaces.  Also suppose
  that there is a probability measure $\pi$ on $S$ such that
  \begin{align*}
    \pi(j)q(j,k)=\pi(k)q'(\phi(k),\phi(j)),\quad
    j,k\in S,
  \end{align*}
  then $\pi$ is the stationary distribution of $X(t)$ and
  $\pi(\phi^{-1}(\cdot))$ is the stationary distribution of $Y(t)$.
\end{theorem}
Theorem~\ref{thm:dynReversedDiscrete} can be stated in a more general
setting, which we now state and prove.

\begin{theorem}
  \label{thm:dynReversedCts}
  Suppose $S$ and $T$ be two locally compact Hausdorff topological
  spaces.  Let $X(t)$ and $Y(t)$ be two stationary Markov jump
  processes with values in $S$ and $T$. Suppose that probability
  semi-group of the process $X(t)$ ($Y(t)$) is characterized by the
  generators $L_X$ ($L_Y$), that is defined over $\dom(L_X)$
  ($\dom(L_Y)$), where the domain is a subset of the Banach space of
  continuous functions over $S$ ($T$) vanishing at infinity, equipped
  with the uniform norm topology. Let $\phi:S\ra T$ be a measure space
  isomorphism such that for all $f\in \dom(L_X)$, we have
  $f\circ\phi^{-1}\in \dom(L_Y)$.  If $\pi$ is a probability measure
  on $S$ such that
  \begin{align*}
    \int_{S}f(x)L_Xg(x)\pi(dx)=\int_T
    L_Y(f\circ\phi^{-1})(y)g\circ\phi^{-1}(y)\phi_{*}\pi(dy),
  \end{align*}
  then $\pi$ is a stationary distribution for $X(t)$.
\end{theorem}
\begin{proof}
  Let $g$ be any element in $\dom(L_x)$. Taking a sequence
  $f_n\in\dom(L_Y)$ such that $f_n$ converges pointwise to the
  constant function $1$, as $n\ra\infty$, we have
  \begin{align*}
    \int_{S}L_Xg(x)\pi(dx)=\int_T
    L_Y(1)(y)g\circ\phi^{-1}(y)\phi_{*}\pi(dy)=0.
  \end{align*}
  By standard results from the theory of positive operator semi-groups
  it is known that for $g\in \dom(L_X)$ implies that the map $x\mapsto
  \EE[g(X(t))|X(0)=x]$ belongs to $\dom(L_X)$ (see Lemma 1.3 of
  \cite{engel1999one} for example). Thus, for $g\in\dom(L_X)$, we have 
  \begin{align*}
    \frac{d}{dt}\int_{S}\EE[g(X(t))|X(0)=x]\pi(dx)=\int_{S}L(\EE[g(X_t)|\cdot])(x)\pi(dx) =0.
  \end{align*}
  It is also known that $\dom(L_X)$ is dense on the space of
  continuous functions vanishing at infinity (Theorem 1.4 of
  \cite{engel1999one}). This implies that
  $\EE_\pi g(X(t))=\EE_\pi g(X(0))$ for all bounded continuous
  functions and all $t>0$, and so, $\pi$ must be a stationary
  measure of $X(t)$.
\end{proof}
If two processes satisfy the hypothesis of the above theorem, we say
that the processes are dynamically reversible.

\section{Some Additional Global Notation}
\label{sec:some-more-global}

In the following few sections we need some useful universal notation
to define the transitions in the Markov processes.  We collect them
here in this section.

Let $\gamma=(x_1,\ldots, x_n)$, $n\in \bbN$ and $x$ respectively be a
list of elements and a particular element belonging to the same
abstract space $S$.  We define the following operators:
\begin{enumerate}
\item Let $\gamma \ftl_ix$, $i=0,\ldots,n$, denote the insertion of
  element $x$ after the $i$-th element in $\gamma$, i.e.,
  $$\gamma\ftl_ix=(x_1,\ldots,x_i,x,x_{i+1},\ldots,x_n).$$
\item Let $\gamma\ftr_i$, $i=1,\ldots,n$, denote the removal of the
  $i$-th element of $\gamma$, i.e.,
  $$\gamma\ftr_i=(x_1,\ldots,x_{i-1}, x_{i+1},\ldots,x_n).$$
\item Let $\gamma\ftu_ix$ denote the replacement of the $i-th$ element
  in $\gamma$ with $x$, i.e.,
  $$\gamma\ftu_ix=(x_1,\ldots,x_{i-1},x,x_{i+1},\ldots,x_n).$$
\item In the above notation, we may drop the subscript $i$ if
  $i=|\gamma|$, i.e., when we are making changes to the last element.
\end{enumerate}

\section{Continuous-Time FCFS Bipartite Matching Model with Reneging}
\label{sec:cont-time-bipart}
In this section, we illustrate how dynamic reversibility is used in
the proof of Theorem~\ref{thm:invMeasure}, by working on a countable
state space Markov model.  This allows us to organize and present the
main ideas without the complexity of dealing with measure valued
processes.

Specifically, in this section, we consider the following modified
version of the First-come-first-serve bipartite matching model
considered in \cite{Adan2018}.  Consider two finite sets of types
$\mathcal{C}=\{c_1,\ldots, c_I\}$ and
$\mathcal{S}=\{s_1,\ldots, s_J\}$ and a bipartite compatibility graph
$G=(\mathcal{C},\mathcal{S},\cE)$ with
$\cE\subset \mathcal{C}\times\mathcal{S}$.  Let $\blambda$ be a
measure on $C\cup S$, and $\mu>0$ be a parameter. We say that $c$ and
$s$ can be \emph{matched} together or are \emph{compatible} if
$(c,s)\in\cE$ in the compatibility graph $E$.  We define the
first-come-first-serve bipartite matching model with reneging as a
Markov jump process with state space, $\Gamma$, which is the set of
all finite ordered lists of elements from $C\cup S$ such that for
every $c\in C$ and $s\in S$ in the list, $(c,s)\notin \cE$.  Further,
given that the state of the process at any time $t$ is
$\gamma=(x_1,\ldots,x_n)$, the state is updated with the following
transition rates:
\begin{enumerate}
\item A new element $x\in C\cup S$ arrives at rate $\lambda(x)$.  At the
  time of the arrival, if there is one or more elements in $\gamma$
  that is compatible to $x$, then the first such element, $x_i$, is
  removed, and we say that $x$ and $x_i$ are matched.  If there is no
  such element, then $x$ is added to the end of the list $\gamma$.
\item Each element in the list is removed at rate $\mu>0$.
\end{enumerate}

The comments and results of Sections~\ref{sec:exis-uniq-stea-stat} and
\ref{sec:prod-form-char} can be mirrored in this setting.  We briefly
review them here.

We can simulate the above process by using arrival from a Poisson
point process $\Phi$ on $(C\cup S)\times\bbR$, with i.i.d.
exponential marks in $\bbR^+$, and with intensity
$\lambda\otimes\ell$.  The base of the Poisson point process $\Phi$
encodes the arrivals of the agents, and the mark of a point encodes
the time each agent is willing to wait (its \emph{patience}), if they
are accepted.  We will use the following notation: for any point
$x\in (C\cup S)\times \bbR\times \bbR^+$, $c_x$ will denote its
projection onto $C\cup S$, $b_x$ will denote the second coordinate,
and $w_x$ will denote the third coordinate.

Standard coupling or Lyapunov based arguments can be used to show that
this Markov process has a stationary regime.  Moreover, a stationary
version of the process can be constructed by using a coupling from the
past scheme that uses an ergodic arrival process, $\Phi$, which is now
a Poisson point process on $(C\cup S)\times \bbR$, with marks as
above.  To construct the stationary regime, the notion of regeneration
time of the system may be defined as follows.  We say that $t\in \bbR$
is called a regeneration time of $\Phi$ if for all $x\in \Phi$ with
$b_x< t$, we have $t-b_x>w_x$.  The forward-time construction of the
process starting from a regeneration time with empty initial
conditions is clear.  Moreover, if $t_1<t_2$ are two regeneration
times and $\eta^1$ and $\eta^2$ are such processes starting from $t_1$
and $t_2$ respectively, then for $t>t_2$, $\eta^1_{t}=\eta^2_t$.
Thus, we can construct a bi-infinite stationary version,
$\{\eta_t\}_{t\in\bbR}$, of this process as a factor of $\Phi$, if we
can find a sequence of regeneration times going to $-\infty$.  Indeed,
if $t_1>t_2>\cdots$ is such a sequence (and set $t_0=\infty$), then
the $\eta_t$ for $t\in[t_i,t_{i-1})$, and for some $i\in\bbN$, is
obtained by simulating the process starting from empty initial
conditions from time $t_i$ until time $t$.  An argument similar to the
Lemma~\ref{lem:boundedRegeneration} can used to show the existence of
such a sequence of regeneration times.

This coupling from the past scheme gives the definition of the
matching function, $$m:\Phi\ra (C \cup S)\times \bbR,$$ similar to the one
defined in Section~\ref{sec:prod-form-char}.  For $x\in \Phi$, let
$T\in\bbR^-$, be a regeneration time before $b_x$.  The value of
$m(x)$ can be set by simulating the process using $\Phi$, starting
from time $T$, with empty initial conditions.  If $x$ is matched to an
agent $y\in \Phi$, then $m(x)=(c_y,b_y)$.  Otherwise, if $x$ reneges, then $m(x)=(c_x,b_w+w_x)$.

Given the function $m$, we can obtain the process $\eta_t$, since
$\eta_t=(x\in \Phi: b_x\leq t< b_{m(x)})$, where the agents in the
list are ordered according to their birth-times $b_\cdot$.  Let $\ttm$
and $\ttu$ be additional marks, referring to whether an agent is
matched or unmatched respectively.  Consider the following detailed
stochastic process $\{\hat\eta_t\}_{t\in\bbR}$: for $t\in\bbR$,
\begin{enumerate}
\item Let $T_t=\min\{b_x:x\in \Phi, b_x\leq t< b_{m(x)}\}$.
\item Let
  $\Gamma_\ttu=\{(c_x,b_x,\ttu):x\in \Phi, T_t\leq b_x\leq t<
  b_{m(x)}\}$ and
  $\Gamma_\ttm=\{(c_x,b_{m(x)},\ttm)\in N: b_x\leq t, T_t\leq
  b_{m(x)}\leq t \}$.
\item Define
  $\hat\eta_t:=((c_x,s_x):x\in \Gamma_\ttu\cup \Gamma_\ttm)$, where
  $s_y$ refers to the matched or unmatched status of an agent
  $y\in\Gamma_\ttu\cup\Gamma_\ttm$, and the list is ordered according
  to their second coordinates, $b_{(\cdot)}$.
\end{enumerate}
Clearly, $\eta_t$ can be obtained from $\hat\eta_t$ by removing the
agents with marks $\ttm$.  We call the process $\hat\eta_t$ the
\emph{Backward detailed process}, following the terminology in \cite{Adan2018}.

Since the backward detailed process at time $t$ only depends on the
points of $\Phi$ before time $t$, it is a stationary process.
Moreover, it is a stationary version of some Markov process since for
$t<s$, the state at time $s$ can be constructed using the state at
time $t$ and the process $\Phi$ in the interval $(t,s]$.  We describe
this Markov process in detail now.  A valid state of this Markov
process can be a finite list of elements, $(x_1,\ldots,x_n)$ from the
set $(C\cup S)\times \{\ttu,\ttm\}$ such that
\begin{enumerate}
\item $s_{x_1}=\ttu$, if $n\geq 1$.
\item If $s_{x_i}=s_{x_j}=\ttu$ then $(c_{x_i},c_{x_j})\notin \cE$.
\item For all $i<j$, if $s_i=\ttu$ and $s_j=\ttm$ then
  $(c_{x_i},c_{x_j})\notin \cE$.
\end{enumerate}

Below, we utilize the definitions in
Section~\ref{sec:some-more-global}.  Additionally, mirroring our
notation in the continuum setting, we define for any $x\in C\cup S$,
let $N(\{x\})=\{y\in C\cup S: (x,y)\in \cE\textrm{ or }(y,x)\in \cE\}$
and for any $A\subseteq C\cup S$ let $N(A)=\cup_{x\in A} N(\{x\})$.
With an abuse of notation, we will let $N(x):=N(\{x\})$ for
$x\in C\cup S$.  Also, for $\gamma \in O(C\cup S,\{\ttu,\ttm\})$ and any
$x\in\gamma$, we will denote
\begin{itemize}
\item $\gamma^x=\{y\in\gamma:y<_\gamma x\}$,
\item $\gamma^\ttu=\{y\in\gamma:s_y=\ttu\}$ and
  $\gamma^\ttm=\{y\in \gamma:s_y=\ttm\}$,
\item $$W_x=
  \begin{cases}
    N(c_x)\backslash N(\gamma^{\ttu,x}) &\textrm{ if } s_x=\ttu,\\
    \emptyset & \textrm{otherwise}.
  \end{cases}
$$
\end{itemize}
The transitions and transition rates for the backward detailed process
are the following: Given that the state of system is
$\hat\eta=(x_1,\ldots, x_n)$,
\begin{enumerate}
\item any agent $x_i\in \hat\eta$, with $s_{x_i}=\ttu$, loses
  patience.  In this case, the new state is
  $\hat\eta'=\hat\eta\ftr_i\ftl(c_{x_i},\ttm)$, except possibly when
  $i=1$, where we need to prune all the leading matched and exchanged
  elements from $\hat\eta'$.  We still denote the new state by
  $\hat\eta\ftr_i\ftl(c_{x_i},\ttm)$, even in this case, keeping in
  mind the all leading matched terms must be removed.  Each such
  transition occurs at rate $\mu$.
\item a new agent $x=(c_x,\ttu)$, $c_x\in C\cup S$, arrives and is
  matched to the agent $x_i\in \hat\eta$, with $(c_{x_i}, c_x)\in\cE$
  and $s_{x_i}=\ttu$.  In this case, the new state is (a valid pruning
  of) $\hat\eta\ftu_i(c_x,\ttm)\ftl(c_{x_i},\ttm)$.  This occurs at
  rate $\blambda(c_x)\1\left (c_x\in W_{x_i} \right)$.
\item a new agent $x=(c_x,\ttu)$, with $c_x\in C\cup S$ arrives and is
  not matched to any agent.  The new state is $\hat\eta\ftl x $.  This
  occurs at rate $\blambda(c_x) \1(c_x\notin N(\hat\eta^\ttu))$.
\end{enumerate}

We now define the \emph{Forward detailed process}, which is the dual
of the process $\{\hat\eta_t\}_{t\in\bbR}$, that we denote by
$\{\check\eta_t\}_{t\in\bbR}$.  For $t\in\bbR$, define
\begin{enumerate}
\item Let $Y_t=\max\{b_{m(x)}:x\in \Phi, b_x\leq t< b_{m(x)}\}$.
\item Define
  $\Xi_\ttm=\{(c_x,b_{m(x)},\ttm):x\in \Phi, b_x\leq t<b_{m(x)} \}$,
  and $\Xi_\ttu=\{(c_x,b_x,\ttu):t< b_x<Y_t, t<b_{m(x)}\}$. 
\item Define $\check\eta_t=((c_x,s_x):x\in\Xi_u\cup \Xi_m)$, where the
  elements are ordered according to the second coordinates
  $b_{(\cdot)}$.
\end{enumerate}

The forward detailed process is also a stationary version of a Markov
process.  Any valid state, $(x_1,\ldots, x_n)$, of this Markov process
of the system satisfies:
\begin{enumerate}
\item $s_{x_n}=\ttm$, when $n\geq 1$.
\item If $s_{x_i}=s_{x_j}=\ttm$, then
  $(c_{x_i},c_{x_j})\notin\cE$. 
\item If $i<j$, $s_{x_i}=\ttu$, $s_{x_j}=\ttm$, then
  $(c_{x_i},c_{x_j})\notin \cE$.
\end{enumerate}

The transitions and the transition rates of the Markov process are
given as follows: Given that the state of the system is $\check\eta$,
the next jump occurs at rate
$\blambda(C\cup S)+Q^0_\ttm(\check\eta)\mu$, where
$Q^i_\ttm(\check\eta)$ is the number of matched elements in
$\check\eta$ after $i$-th location.  Intuitively, this is so because
the total rate of new arrivals is $\blambda(C\cup S)$ and the total
death rate is $Q^0_\ttm\mu$, since $Q^0_\ttm$ is the number of unmatched
agents in the forward process.  For the sake of brevity, let us denote
$\blambda(C\cup S)+n\mu$ by $\rho(n)$, for all $n\in \bbN$.  If
$\check\eta$ is non-empty, at each jump, to obtain the new state, we
need to process the first element, $x_1$, in $\check\eta$.  This is
done with the following probabilities:
\begin{enumerate}
\item If $x_1$ is matched, then it is removed from $\check\eta$.  The
  new state is $\check\eta\ftr_1$.
\item If $x_1$ is unmatched, then for the next state, we sample a
  random variable $\tau\in\bbN$ with distribution
  \begin{align*}
    \PP(\tau=k)=\frac{\mu}{\rho(Q^k_\ttm+1)}\prod_{i=1}^{k-1}\frac{\rho(Q^i_\ttm)}{\rho(Q^i_\ttm+1)},
  \end{align*}
  and then sample $x_{n+1},\ldots x_{\tau}$ i.i.d.  random unmatched elements
  in $\{C\cup S\}$ with distribution
  $\blambda(\cdot)/\blambda(C\cup S)$.
  \begin{enumerate}
  \item If there is a FCFS matching $x_i$, $2\leq i\leq \tau$, then
    set the new state to
    $(x_1^{\max(n,\tau)})\ftu_i(c_{x_1},\ttm)\ftr_1$, with the
    understanding that all the ending unmatched agents are discarded.
  \item If there is no FCFS matching, set the new state to
    $(x_1^{\max(n,\tau)})\ftl_\tau(c_{x_1},\ttm)\ftr_1$.
  \end{enumerate}
\end{enumerate}

If the state is $\check\eta=\emptyset$, then the next jump occurs at
rate $\blambda(C\cup S)$.  When a jump occurs, a random unmatched
point $x_1\in C\cup S$ is sampled with distribution
$\blambda(\cdot)/\blambda(C\cup S)$.  The next state is decided as in
step (2) above, with this new $\check\eta=x_1$, we ignore the details
here.

We have the following theorem.
\begin{theorem}
  The two processes $\hat\eta_t$ and $\check\eta_t$ are dynamically
  reversible.  The concerned isomorphism $\phi$ is the one that takes
  a valid state $\hat\eta$, reverses the order of its elements and
  flips the marks $\ttu$ and $\ttm$.  The stationary distribution is
  \begin{align*}
    \pi(\hat\eta)&=K\1(\hat\eta\textrm{ is valid}) \prod_{i=1}^{n}\frac{\lambda(c_{x_i})}{\rho(Q^i_\ttu)}\\
    \Pi(\emptyset)&=K,
  \end{align*}
  where $\hat\eta=(x_1,\ldots,x_n)$ and where $K$ is a normalizing constant.
\end{theorem}
\begin{proof}[Outline of the proof]
  To prove this, we start by looking at the balance equations of the
  form in Theorem~\ref{thm:dynReversedDiscrete}.  Let $\hat\eta$ be
  the state of the backward detailed process, and let $\hat\eta'$ be a
  state after a valid transition.  In the following, we illustrate the
  local balance condition eq.~\ref{eq:local-balance-1} for only one
  type of transition.  Other kinds of transitions can be handled
  similarly.  Let $q$ and $q'$ be the transition rates of the backward
  and forward detailed processes respectively.  Also, let
  $c_j(\gamma)$ denote the type of the $j$-th element in $\gamma$ for
  any $\gamma\in O(C\cup S)$.

  Suppose that $\hat\eta=(x_1,\ldots,x_n)$, $n>0$, and that
  $\hat\eta'$ is obtained from $\hat\eta$ when one of the elements at
  $x_i$, at some location $i>1$, is matched and exchanged with a new
  arrival $(c_x,\ttu)$.  In this case,
  $\hat\eta'= \hat\eta\ftu_i(c_x,\ttm)\ftl(c_{x_i},\ttm)$ and
  \begin{align*}\numberthis\label{eq:dis2}
    \frac{\pi(\hat\eta)}{\pi(\hat\eta')}q(\hat\eta,\hat\eta')
    &=\lambda(c_x)\prod_{j=1}^{|\hat\eta|}\frac{\lambda(c_j(\hat\eta))}{\rho(Q^j_\ttu(\hat\eta))}\prod_{j=1}^{|\hat\eta'|}\frac{\rho(Q^j_\ttu(\hat\eta'))}{\lambda(c_j(\hat\eta'))}.
  \end{align*}
  The first $i-1$ elements in $\hat\eta'$ are $x_1,\ldots,x_{i-1}$ and
  $|\hat\eta'|=|\hat\eta|+1$.  Moreover,
  $Q^j_\ttu(\hat\eta)=Q^j_\ttu(\hat\eta')+1$ for
  $i\leq j\leq |\hat\eta|$.  Hence, eq.~\ref{eq:dis2} simplifies to
  \begin{align*}
    \frac{\pi(\hat\eta)}{\pi(\hat\eta')}q(\hat\eta,\hat\eta')
    &=\prod_{j=i}^{|\hat\eta|}\frac{\rho(Q^j_\ttu(\hat\eta'))}{\rho(Q^j_\ttu(\hat\eta')+1)}\times\rho(Q^{|\hat\eta'|}_\ttu(\hat\eta))\\
    &=\PP(\tau_{\phi(\hat\eta')}> |\phi(\hat\eta')|-i)\times \rho(Q^0_\ttm(\phi(\hat\eta')))\\
    &=q'(\phi(\hat\eta'),\phi(\hat\eta)).
  \end{align*}
  We claim that local balance equations for other valid transitions
  can also be handled similarly.  This completes the proof
  this theorem.
\end{proof}


\section{Proof of Theorem \ref{thm:invMeasure}}
\label{sec:calculations}
In this section we present detailed calculations to show that the
backward detailed process, $\{\hat\eta_t\}_{t\in\bbR}$, and the
forward detailed process, $\{\check\eta_t\}_{t\in\bbR}$, defined in
Section~\ref{sec:prod-form-char}, are dynamically reversible as jump
Markov process.  In turn, we are able to prove
Theorem~\ref{thm:invMeasure}.

We first define the valid states of the Markov process corresponding
to the forward detailed process, and present the transitions and the
transition rates, since these were skipped in the discussion in
Section~\ref{sec:prod-form-char}.

A valid state of the forward detailed process is given by the
following rules.
\begin{definition}
  $(x_1,\ldots,x_n)\in O(D\times\mbf{C}\times\{\ttu,\ttm\})$ is a
  \emph{valid} state of the forward detailed process if
  \begin{enumerate}
  \item $s_{x_n}=\ttm$, if $n\geq 1$,
  \item For all $i,j$, if $s_{x_i}=s_{x_j}=\ttm$ and
    $d(p_{x_i},p_{x_j})\leq 1$, then $c_{x_i}= c_{x_j}$,
  \item For all $i<j$, if $s_{x_i}=\ttu$, $s_{x_j}=\ttm$ and
    $d(p_{x_i},p_{x_j})\leq 1$, then $c_{x_i}= c_{x_j}$.
  \end{enumerate}
\end{definition}
Condition 2 in the above definition essentially states that there cannot
be a compatible matched pair in a valid state.  This condition is
equivalent to the condition that
$$\{y\in x_1,\ldots x_n:s_y=\ttm\}\cap N(\{y\in x_1,\ldots
x_n:s_y=\ttm\})=\emptyset.$$ This is because, if there was a violating
pair at time $t$, such a pair could have potentially matched to each
other before time $t$ instead of matching to their present
matches. Condition 3 is required since otherwise a violating pair $x_i$
and $x_j$, $i<j$, the particle $x_j$ could potentially have matched
with the particle $x_i$ instead, which arrives before the particle
$x_j$ is matched to.  This condition is equivalent to the condition
that for all $1\leq j\leq n$,
$$s_{x_j}=\ttm\implies x_j\notin N(\{y\in x_1,\ldots,
x_i:s_y=\ttu\}).$$

The Markov process corresponding to $\{\check\eta_t\}$ evolves as
follows.  Let $\check\eta=(x_1,\ldots,x_n)$, $n=|\check\eta|$, be the
state of the system at some time $t$.  The next jump occurs at rate
$\rho(Q^0_\ttm)=2\lambda(D)+Q^0_\ttm\mu$.  If $\check\eta$ is
non-empty, the first element, $x_1$, in the list $\check\eta$ is
processed at the next jump according to the following rules.
\begin{enumerate}
\item If $x_1$ is matched, then it is removed from $\check\eta$.  The
  new state is $\check\eta\ftr_1$.
\item If $x_1$ is unmatched, then for the next state, we sample a
  random variable $\tau\in\bbN$ with distribution
  \begin{align*}
    \PP(\tau=k)=\frac{\mu}{\rho(Q^k_\ttm+1)}\prod_{i=1}^{k-1}\frac{\rho(Q^i_\ttm)}{\rho(Q^i_\ttm+1)},
  \end{align*}
  and then sample $x_{n+1},\ldots x_{\tau}$ i.i.d.  random unmatched
  points in $D\times\mbf{C}$ with distribution
  $\lambda\otimes m_c/(2\lambda(D))$.
  \begin{enumerate}
  \item If there is a first-in-first-match $x_i$, $2\leq i\leq \tau$,
    for $x_1$, then set the new state to
    $(x_1^{\max(n,\tau)})\ftu_i(p_{x_1},c_{x_1},\ttm)\ftr_1$, with the
    understanding that all the unmatched particles at the end of the
    list are discarded.
  \item If there is no FIFM matching, then set the new state to
    $(x_1^{\max(n,\tau)})\ftl_\tau(p_{x_1}c_{x_1},\ttm)\ftr_1$.
  \end{enumerate}
\end{enumerate}
If the state is $\check\eta=\emptyset$, then the next jump occurs at
rate $\rho(0)=2\lambda(D)$.  A random unmatched particle $x_1$ is sampled
from with distribution $\blambda/2\lambda(D)$.  The next
state is decided as in step (2) above with $\check\eta=(x_1)$.

We are now ready to prove Theorem~\ref{thm:invMeasure}. 

\begin{proof}[Proof of Theorem~\ref{thm:invMeasure}]
  To obtain the stationary distribution, we check the conditions of
  Theorem~\ref{thm:dynReversedCts}. The space
  $O(D,\mbf{C}\times\{\ttu,\ttm\})$ is viewed as a subset of
  $\sqcup_{n=0}^\infty (D\times \mbf{C}\times\{\ttu,\ttm\})^n$, and we
  use the induced topology on $O(D,\mbf{C}\times\{\ttu,\ttm\})$.  With
  this topology, $O(D,\mbf{C}\times\{ttu,\ttm\})$ is a locally compact
  Hausdorff space. Let $\hat D=D\times\mbf{C}\times\{\ttu,\ttm\}$ and
  let $\hat \lambda$ be the measure $ \lambda\otimes m_c\otimes m_c$
  on $\hat D$.  The probability semi-group of the processes $\hat\eta$
  acts over the Banach space of continuous functions that vanish at
  infinity, where we use the uniform norm topology. Moreover, the
  generator of $\hat\eta$ can at least be defined on the space of
  compactly supported continuous functions, and has the form:
  \begin{equation}
    \begin{aligned} \label{eq:generatorOne}
  L_1 f(\hat\eta)=
  & \sum_{i=1}^{|\hat\eta|}\1(s_{x_i}=\ttu)\mu [f(\hat\eta\ftr_i\ftl(p_{x_i},c_{x_i},\ttm))-f(\hat\eta)]\\
  &\quad+\int_{\tilde{D}}\sum_{i=1}^{|\hat\eta|}\1(x\in W_{x_i})[f(\hat\eta\ftu_i(p_x,c_x,\ttm)\ftl(p_{x_i},c_{x_i},\ttm))-f(\hat\eta)]\blambda(dx)\\
  &\quad+\int_{\tilde{D}}\1(x\notin
  N(\hat\eta^\ttu))[f(\hat\eta\ftl(p_x,c_x,\ttu))-f(\hat\eta)]\blambda(dx).
  \end{aligned}
\end{equation}

Similarly, it can also be seen that the generator of $\check \eta$ can
also be defined over the space of compactly supported continuous
functions, and the value of the generator $L_2g(\check\eta)$ is the
sum of the following terms (in the order of the transitions listed earlier):
\begin{itemize}
\item (1):
  $\rho(Q^0_\ttm(\check\eta))\1(s_{x_1}=\ttm)
  [g(\check\eta\ftr_1)-g(\check\eta)]$,
\item (2a):
  \begin{align*}
    &\rho(Q^0_\ttm(\check\eta))\1(s_{x_1}=\ttu)\\
    &\IndI\times\left(\sum_{k=2}^{|\check\eta|}
      \PP(\tau(\check\eta)>k)\1((p_{x_1},c_{x_1})\in W_{x_k})[g(\hat\eta\ftu_k(p_{x_1},c_{x_1},\ttm)\ftr_1)-g(\hat\eta)]\right.\\
    &\left.\IndII+\sum_{k=|\check\eta|+1}^\infty\PP(\tau(\check\eta)>k)\int_{\tilde{D}^{k-|\check\eta|}}\left(\vphantom{x_1^k}\1((p_{x_1},c_{x_1})\in
      W_{x_k})\right.\right.\\
    &\left.\left.\IndIII\times[g((x_1^{k})\ftu_k(p_{x_1},c_{x_1},\ttm)\ftr_1)-g(\hat\eta)]\right)\blambda(dx_{|\check\eta|+1}^k)\vphantom{\sum_{k=2}^{|\check\eta|}}\right),
  \end{align*}
  where we have set $s_{x_j}=\ttu$ for all $j>|\check\eta|$.
\item (2b):
  \begin{align*}
    &\rho(Q^0_\ttm(\check\eta))\1(s_{x_1}=\ttu)\\
    &\IndI\times\left(\sum_{k=1}^{|\check\eta|} \PP(\tau = k)\1((p_{x_1},c_{x_1})\notin N(\check\eta_1^{k,\ttu}))[g(\hat\eta\ftl_k(p_{x_1},c_{x_1},\ttm)\ftr_1)-g(\hat\eta)]\right.\\
    &\left.\IndII+\sum_{k=|\check\eta|+1}^\infty \PP(\tau =
      k)\int_{\tilde{D}^{k-|\check\eta|}}\1((p_{x_1},c_{x_1})\notin
      N(x_1^{k,\ttu}))\right.\\
    &\left.\IndIII\times[g((x_1^{k})\ftl(p_{x_1},c_{x_1},\ttm)\ftr_1)-g(\check\eta)]\blambda(dx_{|\check\eta|+1}^k)\right),
  \end{align*}
  where $s_{x_j}=\ttu$ for all $j>|\check\eta|$.
\end{itemize}

When $\check\eta=\emptyset$, we have the following terms in
$L_2g(\emptyset)$.
\begin{itemize}
\item (2a.):
  \begin{align*}
    &\rho(0)\int_{\tilde{D}}\sum_{k=2}^\infty\PP(\tau(\emptyset)>k-1)\int_{\tilde{D}^{k-1}}\left(\vphantom{x_1^k}\1((p_{x_1},c_{x_1})\in
      W_{x_k})\right.\\
    &\IndIII\left.\times[g((x_2^{k-1})\ftl(p_{x_1},c_{x_1},\ttm))-g(\emptyset)]\right)\blambda(dx_{2}^k)\blambda(dx_1),
  \end{align*}
  where, $s_{x_j}=\ttu$ for all $j>0$.
\item (2b.):
  \begin{align*}
    &\rho(0)\int_{\tilde{D}}\sum_{k=1}^\infty
      \PP(\tau_1=k)\int_{\tilde{D}^{k-1}_\ttu}\left(\vphantom{x_1^k}\1((p_{x_1},c_{x_1})\notin N(x_1^{k,\ttu}))\right.\\
    &\IndIIII\left.\times[g((x_2^{k})\ftl(p_{x_1},c_{x_1},\ttm))-g(\emptyset)]\right)\blambda(dx_{2}^k)\blambda(dx_1)
  \end{align*}
  where, $s_{x_j}=\ttu$ for all $j>0$.
\end{itemize}

For the product form distribution, we check the balance condition of
Theorem~\ref{thm:dynReversedCts}, with an appropriate measure space isomorphism
$\phi$.  The isomorphism is given by the function $\rx$, defined in
Section~\ref{sec:prod-form-char}, that reverses the order and flips
marks $\ttu$ and $\ttm$ of a valid state.  In the following, let
$\phi:=\rx$ denote this function.

Taking any two compactly supported continuous functions $f,g$, and
taking each term of $\int gL_1f+gf\hat\pi(d\hat\eta)$, we show that it
corresponds to a few terms in
$\int
L_2(g\circ\phi^{-1})(\phi(\hat\eta))f(\hat\eta)+gf\hat\pi(d\hat\eta)$,
so that the sum of these expressions is equal.  In particular, the
following steps suffice.
\begin{enumerate}
\item Take the first summation term of $\int gL_1 f+gf\ d\hat\pi$ when it
  is expanded using eq.~\ref{eq:generatorOne}.  Take $i$-th term, with
  $i>1$.  Set $\hat\eta'=\hat\eta\ftr_i\ftl(p_i,c_i,\ttm)$.  Let
  $n:=|\hat\eta|$, and so $n=|\hat\eta'|$.  Also, let
  $\hat\eta=(x_1,\ldots,x_{n})$ and $\hat\eta'=(x_1',\ldots,x_n')$.
  The corresponding term is
  \begin{equation}
    \begin{aligned}\label{eq:test1}
      &\int \mu\1(s_{x_i}=\ttu)
      g(\hat\eta)f(\hat\eta')\hat\pi(d\hat\eta)\\
      &:=\mu\sum_{n=i}^\infty\int_{\hat{D}^{n}}\1(s_{x_i}=\ttu)\hat\pi(\hat\eta)g(\hat\eta)f(\hat\eta')\hat\lambda(d\hat\eta)\\
      &=\sum_{n=i}^\infty\int_{\hat{D}^n}\hat\pi(\hat\eta')\frac{\mu\hat\pi(\hat\eta'\ftl_{i-1}(p_{x_{n}'},c_{x_{n}'},\ttu)\ftr)}{\hat\pi(\hat\eta')}g(\hat\eta'\ftl_i(p_{x_{n}'},c_{x_{n}'},\ttu)\ftr)f(\hat\eta')\hat\lambda(d\hat\eta')\\
      &=\int \1(n\geq
      i)\PP(\tau_{\phi(\hat\eta')}=n-i+1)\rho(Q^0_\ttm(\phi(\hat\eta')))g(\hat\eta'\ftl_i(p_{x_n'},c_{x_n},\ttu)\ftr)f(\hat\eta')\hat\pi(d\hat\eta'),
    \end{aligned}
  \end{equation}
  where in the second equality, we have used that
  $\hat\eta=\hat\eta'\ftl_{i-1}(p_{x_{n}'},c_{x_{n}'},\ttu)\ftr$, and
  in the third equality, we use
  \begin{equation}
    \begin{aligned}
      \frac{\mu\pi(\hat\eta'\ftl_{i-1}(p_n',c_n',\ttu)\ftr)}{\pi(\hat\eta')}
      &=\frac{\mu}{\rho(N^i_\ttu(\hat\eta))}\prod_{j=i}^{n-1}\frac{\rho(N^j_\ttu(\hat\eta'))}{\rho(N^{j+1}_\ttu(\hat\eta))}\rho(N^n_\ttu(\hat\eta'))\label{eq:ratio-to-prob}\\
      &=\PP(\tau_{\phi(\hat\eta')}=n-i+1)\rho(N^0_\ttm(\phi(\hat\eta'))).
    \end{aligned}
  \end{equation}
  Similarly, in the first summation term of $\int gL_1f+gf$, taking
  $i=1$, and letting $k(\hat\eta)$ be the maximum element such that
  $x_2,\ldots,x_k$ are all matched in $\hat\eta$, we have
  \begin{align*}
    &\int \mu
      g(\hat\eta)f(\hat\eta')\hat\pi(\hat\eta)\\
    &=\mu\sum_{n=2}^\infty\int_{\hat{D}^n}\sum_{j=1}^{n-1}\1(k(\hat\eta)=j)g(x_1^{n})f(x_{j+1}^n\ftl(p_{x_1},c_{x_1},\ttm))\hat\pi(\hat\eta)\hat\lambda(d\hat\eta)\\
    &\IndI+\mu\sum_{n=1}^\infty\int_{\hat{D}^n}\1(k(\hat\eta)=n)g(x_1^n)f(\emptyset)\hat\pi(\hat\eta)\hat\lambda(\hat\eta),
  \end{align*}
  where we have used the fact that if $k(\hat\eta)=|\hat\eta|$, then
  $\hat\eta'=\emptyset$.  Consider the first term in the above
  expression.  Setting $m=n-j+1$ and
  ${x'}_1^m=(x_{j+1}^{n}\ftl(p_{x_1},c_{x_1},\ttm))$, we obtain:
  \begin{equation}
    \begin{aligned}
      &\mu\sum_{n=2}^\infty\int_{\hat{D}^n}\sum_{j=1}^{n-1}\1(k(\hat\eta)=j)g(x_1^{n})f(x_{j+1}^n\ftl(p_{x_1},c_{x_1},\ttm))\hat\pi(\hat\eta)\hat\lambda(d\hat\eta)\\
      &=\mu\sum_{j=1}^\infty\sum_{m=2}^\infty\int_{\hat{D}^{m}}\1(s_{x_m'}=\ttm)\blambda(\tilde{D})^{j-1}\EE_{X_2^j}\left[g((p_{x_m'},c_{x_m'},\ttu)X_2^j{x'}_1^{m-1})\right.\\
      &\IndIII
      \IndI\left.f({x'}_1^m)\hat\pi((p_{x_m'},c_{x_m'},\ttu)X_2^j{x'}_1^{m-1})\right]\hat\lambda(d{x'}_1^m), \label{eq:test2-1}
    \end{aligned}
  \end{equation}
  where $X_2^j$ are i.i.d.  particles in $\tilde D$ with marks $\ttm$,
  with distribution $\blambda(\cdot)/\blambda(\tilde D)$.  Using a
  similar computation as in eq.~\ref{eq:ratio-to-prob}, it is easy to
  see that RHS of eq.~\ref{eq:test2-2} is
  \begin{equation}
    \begin{aligned}
      &\sum_{m=2}^\infty\int_{\hat{D}^m}\1(s_{x'_m}=\ttm)\rho(Q^0_\ttm(\phi({x'}_1^m)))\sum_{j=1}^\infty\PP(\tau(\phi({x'}_1^m))=m+j-1)\hat\pi({x'}_1^m)\label{eq:test2-2}\\
      &\IndIII
      \IndI\times\EE_{X_2^j}\left[g((p_{x_m'},c_{x_m'},\ttu)X_2^j{x'}_1^{m-1})f({x'}_1^m)\right]\hat\lambda(d{x'}_1^m).
    \end{aligned}
  \end{equation}
  Similarly, the second term in the eq.~\ref{eq:test2-1} is
  \begin{equation}
    \begin{aligned}
      &\mu\sum_{n=1}^\infty\int_{\hat{D}^n}\1(k(\hat\eta)=n)g(x_1^n)f(\emptyset)\hat\pi(\hat\eta)\hat\lambda(\hat\eta)\label{eq:test2-3}\\
      &=\hat\pi(\emptyset)\rho(0)\sum_{j=1}^\infty\PP(\tau(\emptyset)=j)\EE_{X_1^j}f(\emptyset)g(X_1^j),
    \end{aligned}
  \end{equation}
  where, $X_2^j$ are as before, and $X_1$ is an independent sample
  from $\tilde{D}$ with mark $\ttu$.

\item Now, take the second summation in $\int gL_1 f+gf\hat\pi$, and
  the $i$-th term, $i>1$, in that summation.  Using a similar
  computation as in previous item, we obtain:
  \begin{align*}
    &\int\int_{\tilde{D}}\1(s_{x_i}=\ttu)\1\left((p,c)\in
      W_{x_i}\right)f(\hat\eta')g(\hat\eta)\blambda(dp,dc)\hat\pi(d\hat\eta)\\
    &=\sum_{n=i}^\infty\int_{\hat{D}^n}\int_{\tilde{D}}\1(s_{x_i}=\ttu)\1\left((p,c)\in
      W_{x_i}\right)\\
    &\IndIIII
      f(x_1^n\ftu_i(p,c,\ttm)\ftl(p_{x_i},c_{x_i},\ttm))g(x_1^n)\hat\pi(x_1^n)\tilde\lambda(dp,dc)\hat\lambda(d\hat\eta).
  \end{align*}
  Setting
  ${x'}_1^{n+1}=x_1^n\ftu_i(p,c,\ttm)\ftl(p_{x_i},c_{x_i},\ttm)$ and
  $m=n+1$, we have the RHS of the above equation is equal to
  \begin{equation}
    \begin{aligned}\label{eq:test3}
      &\sum_{m=i+1}^\infty\int_{\hat{D}^{m}}\1(s_{x_i'}=\ttm=s_{x_{m}'})\1((p_{x_m'},c_{x_m'})\in
      W_{x_i'})\\
      &\IndIII
      f({x'}_1^m)g({x'}_1^{m-1}\ftu_i(p_{x_m'},c_{x_m'},\ttu))\hat\pi({x'}_1^{m-1}\ftu_i(p_{x_m'},c_{x_m'},\ttu))\hat\lambda(d{x'}_1^m)\\
      &=\sum_{m=i+1}^\infty\int_{\hat{D}^m}\rho(Q^0_\ttm(\phi({x'}_1^m)))\PP(\tau(\phi({x'}_1^m)>m-i))\1(s_{x_i'}=\ttm=s_{x_{m}'})\\
      &\IndIII\1((p_{x_m'},c_{x_m'})\in W_{x_i'})
      f({x'}_1^m)g({x'}_1^{m-1}\ftu_i(p_{x_m'},c_{x_m'},\ttu))\hat\pi({x'}_1^{m})\hat\lambda(d{x'}_1^m).
    \end{aligned}
  \end{equation}
  Similarly, taking first term in the second summation, and letting
  $k(\hat\eta)$ as before, we have
  \begin{equation}
    \begin{aligned}\label{eq:test4-1}
      &\int\int_{\tilde{D}}\1((p,c)\in
      W_{x_1})f(\hat\eta')g(\hat\eta)\tilde\lambda(dp,dc)\hat\pi(d\hat\eta)\\
      &=\sum_{n=2}^\infty\int_{\hat{D}^n}\int_{\tilde{D}}\sum_{j=1}^{n-1}\1(k(\hat\eta)=j)\1((p,c)\in
      W_{x_1})\\
      &\IndIII f(x_{j+1}^n\ftl(p_{x_1},c_{x_1},\ttm))g(x_1^n)\hat\pi(x_1^n)\tilde\lambda(dp,dc)\hat\lambda(dx_1^n)\\
      &+\sum_{n=1}^\infty\int_{\hat{D}^n}\int_{\tilde{D}}\1(k(\hat\eta)=n)\1((p,c)\in
      W_{x_1})f(\emptyset)g(x_1^n)\hat\pi(x_1^n)\tilde\lambda(dp,dc)\hat\lambda(dx_1^n).
    \end{aligned}
  \end{equation}
  Computations similar to the one in eqs.~\ref{eq:test2-2} and
  \ref{eq:test2-3} show that the above is equal to
  \begin{equation}
    \begin{aligned}\label{eq:test4-2}
      &\sum_{m=2}^\infty\int_{\hat{D}^m}\1(s_{x'_m}=\ttm)\rho(Q^0_\ttm(\phi({x'}_1^m)))\sum_{j=1}^\infty\PP(\tau(\phi({x'}_1^m))>m+j-1)\\
      &\IndI\times\EE_{X_2^{j+1}}\1((p_{x_1},c_{x_1})\in
      W_{X_{j+1}})\left[g((p_{x_m'},c_{x_m'},\ttu)X_2^j{x'}_1^{m-1})f({x'}_1^m)\right]\hat\pi({x'}_1^m)\hat\lambda(d{x'}_1^m)\\
      &+\hat\pi(\emptyset)\rho(0)\sum_{j=1}^\infty\PP(\tau(\emptyset)>j)\EE_{X_1^{j+1}}\1((p_{X_1},c_{X_1})\in
      W_{X_{j+1}})f(\emptyset)g(X_1^j),
    \end{aligned}
  \end{equation}
  where $X_1$ and $X_{j+1}$ are i.i.d.  with marks $\ttu$ and $X_2^j$
  are i.i.d with marks $\ttm$.
\item Finally,
  \begin{align*}
    &\int \int_{\tilde{D}}\1((p,c)\notin N(\hat\eta^\ttu))
      g(\hat\eta)f(\hat\eta')\tilde\lambda(dp,dc)\hat\pi(d\hat\eta)\\
    &=\sum_{n=0}^\infty\int_{\hat{D}^n}\int_{\tilde{D}}\1((p,c)\notin
      N(x\in x_1^n:s_{x}=\ttu))\\
    &\IndIII
      g(x_1^n)f(x_1^n\ftl(p,c,\ttu))\hat\pi(x_1^n)\blambda(dp,dc)\hat\lambda(dx_1^n)
  \end{align*}
  Setting $m=n+1$, ${x'}_1^m=x_1^n\ftl(p,c,\ttu)$, we have
  \begin{equation}
    \begin{aligned}
      \label{eq:test5}
      &\sum_{m=1}^\infty \int_{\hat{D}^{m}}
      \1(s_{x_{m}'}=\ttu)g({x'}_1^{m-1})f({x'}_1^m)\hat\pi({x'}_1^{m-1})\hat\lambda(d{x'}_1^m)\\
      &=\sum_{m=1}^\infty\int_{\tilde{D}^{m}}\rho(Q^0_\ttu(\phi({x'}_1^m)))\1(s_{x_m'}=\ttu)g({x'}_1^{m-1})f({x'}_1^m)\hat\pi({x'}_1^{m-1})\tilde\lambda(d{x'}_1^m).
    \end{aligned}
  \end{equation}

\end{enumerate}

Since summing over the RHS of equations~\ref*{eq:test1} to
\ref*{eq:test5}, for all $i\geq 1$, we get
$\int fL_2g+fg\hat\pi(d\eta)$, we conclude that the two Markov
processes $\hat\eta_t$ and $\check\eta_t$ are dynamically reversible
with respect to $\hat\pi$.
\end{proof}

\section{Proof of Lemma \ref{lem:aux}}
\label{sec:proof-lemma}
The proof is by induction on $m$.  For $m=1$, we need to show that
\begin{align}\label{eq:basecase}
  \sum_{\sigma\in P(n,1)}\prod_{i=1}^{n+1}\frac{1}{\alpha_{\sigma_x(i)}+\beta_{\sigma_y(i)}}=\frac{1}{\beta_1}\prod_{i=1}^n\frac{1}{\alpha_i}.
\end{align}
We use induction on $n$ to prove eq.~\ref{eq:basecase}.  For $n=1$,
this is clear, since
\begin{align*}
  \frac{1}{(\alpha_1)(\alpha_1+\beta_1)}+\frac{1}{(\beta_1)(\beta_1+\alpha_1)}=\frac{1}{(\alpha_1)(\beta_1)}
\end{align*}
Let the length of the sequence $\alpha$ be $n$.  Assuming the
inductive hypothesis for eq.~\ref{eq:basecase}, we have
\begin{align*}
  &\sum_{k=1}^{n+1}\sum_{\substack{\sigma\in
    P(n,1),\\\sigma(k)-\sigma(k-1)=(0,1)}}\prod_{i=1}^{n+1}\frac{1}{\alpha_{\sigma_x(i)}+\beta_{\sigma_y(i)}}\\
  &=\left(\sum_{k=1}^{n}\sum_{\substack{\sigma\in
    P(n,1),\\\sigma(k)-\sigma(k-1)=(0,1)}}\prod_{i=1}^{n}\frac{1}{\alpha_{\sigma_x(i)}+\beta_{\sigma_y(i)}}+\prod_{i=1}^{n}\frac{1}{\alpha_i}\right)\frac{1}{\alpha_n+\beta_1}\\
  &=\left(\sum_{\sigma\in P(n-1,1)}\prod_{i=1}^{n}\frac{1}{\alpha_{\sigma_x(i)}+\beta_{\sigma_y(i)}}+\prod_{i=1}^{n}\frac{1}{\alpha_i}\right)\frac{1}{\alpha_n+\beta_1}\\
  &=\left(\prod_{i=1}^{n-1}\frac{1}{\alpha_i}\right)\frac{1}{\alpha_n+\beta_1}\left(\frac{1}{\beta_1}+\frac{1}{\alpha_n}\right)\\
  &=\frac{1}{\beta_1}\left(\prod_{i=1}^{n}\frac{1}{\alpha_i}\right),
\end{align*}
where in the first step we have grouped the first $n$ terms together.
This finishes the proof of the base case (eq.~\ref{eq:basecase}) for
the induction on $m$.

Now, suppose that the length of the sequence $\beta$ is equal to $m$,
$m>1$.  Let $P_k(n,m-1)$, $0\leq k\leq n$, denote the set of paths in
$P(n,m-1)$ where the first $(0,1)$ jump is at location $k$.  We have
the following decomposition of the summation
\begin{align*}
  \sum_{\sigma\in
  P(n,m)}\prod_{i=1}^{n+m}\frac{1}{\alpha_{\sigma_x(i)}+\beta_{\sigma_y(i)}}
  &=\sum_{k=1}^{n+1}\sum_{P_k(n,m-1)}
    \sum_{r=0}^{k-1}\prod_{i=1}^{n+m}\frac{1}{\alpha_{\sigma_x^r(i)}+\beta_{\sigma^r_y(i)}},
\end{align*}
where $\sigma^r$ denotes the path obtained by inserting a $+(0,1)$
jump in the $r$-th location of $\sigma$.  Thus,
\begin{align*}
  \sum_{\sigma\in P(n,m)}\prod_{i=1}^{n+m}\frac{1}{\alpha_{\sigma_x(i)}+\beta_{\sigma_y(i)}}
  &=\sum_{k=1}^{n+1}\sum_{P_k(n,m-1)}\sum_{r=0}^{k-1}\prod_{i=1}^{k}\frac{1}{\alpha_{\sigma^r_x(i)}+\beta_{\sigma^r_y(i)}}\prod_{i=k+1}^{n+m}\frac{1}{\alpha_{\sigma^r_x(i)}+\beta_{\sigma^r_y(i)}}
\end{align*}
The inner-most summation in the above expression is the $(n,1)$ case
of this lemma, and therefore, by the induction hypothesis, we get that
the
\begin{align*}
  \prod_{i=k+1}^{n+m}\frac{1}{\alpha_{\sigma^{k}_x(i)}+\beta_{\sigma^k_y(i)}}\times\frac{1}{\beta_1}\prod_{i=1}^{k-1}\frac{1}{\alpha_i}=\frac{1}{\beta_1}\prod_{i=1}^{n+m-1}\frac{1}{\alpha_{\sigma_x(i)}+\beta'_{\sigma_y(i)}},
\end{align*}
where $\beta'$ is the sequence of length $m-1$ with
$\beta'_{i}=\beta_{i+1}$, for $1\leq i\leq m-1$.  Therefore,
\begin{align*}
  \sum_{\sigma\in P(n,m)}\prod_{i=1}^{n+m}\frac{1}{\alpha_{\sigma_x(i)}+\beta_{\sigma_y(i)}}
  &=\frac{1}{\beta_1}\sum_{k=0}^n\sum_{P_k(n,m-1)}\prod_{i=1}^{n+m-1}\frac{1}{\alpha_{\sigma_x(i)}+\beta'_{\sigma_y(i)}}\\
  &=\frac{1}{\beta_1}\sum_{\sigma\in
    P(n,m-1)}\prod_{i=1}^{n+m-1}\frac{1}{\alpha_{\sigma_x(i)}+\beta'_{\sigma_y(i)}}\\
  &=\prod_{i=1}^n\frac{1}{\alpha_i}\prod_{i=1}^m\frac{1}{\beta_i},
\end{align*}
by the induction hypothesis.  This completes the proof the lemma.

\section{Proof of Theorem \ref{thm:fkg-same-type}}
\label{sec:proof-theorem-fkg}
Let $A=\xi\backslash\gamma$, $B=\gamma\backslash \xi$ and
$C=\xi\cap \gamma$.  The statement of the theorem is equivalent to
showing that
\begin{align}
  \label{eq:fkg-disjoint}
  \tPi(A\cup B\cup C)\tPi(C)\geq \tPi(A\cup C)\tPi(B\cup C).
\end{align}
Let $\bar C$ denote another copy of the $C$, where we add over-lines
to the particles to distinguish them from particles of $C$.  Also, for
any $E,\gamma\subseteq D\times\mbf{C}$, let
$N_E(\gamma)=N(\gamma\cap E)$.

Using the auxiliary Lemma~\ref{lem:aux}, the LHS of the inequality
above can be expressed as
\begin{equation}
\begin{aligned}
  &\tilde \Pi(A\cup B\cup C)\tilde\Pi(\bar C)\\
  &=\sum_{\substack{\mbf{e}\in \mathcal{P}(A\cup B\cup C)\\ \mbf{d}\in\mathcal{P}(\bar{C})}}\prod_{i=1}^{n+m+k}\frac{1}{\blambda(N(\mbf{e}_1^i))+i\mu}\prod_{j=1}^k\frac{1}{\blambda(N(\mbf{d}_1^j))+j\mu}\\
  &=\sum_{\substack{\mbf{e}\in \mathcal{P}(A\cup B\cup C)\\ \mbf{d}\in\mathcal{P}(\bar{C})}}\sum_{\sigma\in
  P(n+m+k,k)}\prod_{i=1}^{n+m+2k}\frac{1}{\blambda(N(\mbf{e}_{1}^{\sigma_x(i)}))+\blambda(N(\mbf{d}_{1}^{\sigma_y(i)}))+i\mu},\label{eq:fkg-1}
\end{aligned}
\end{equation}
where we use $\sigma_x(i),\ \sigma_y(i)$ to represent the first and
second coordinates of $\sigma(i)$.

Since
\begin{align}\label{eq:fkg-volu-ineq}
  \blambda(N_{A\cup B\cup C}(\gamma))\leq\left(\substack{
  \blambda(N_{A}(\gamma))+\blambda(N_{B}(\gamma))+\blambda(N_{C}(\gamma))\\-\blambda(N_{A}(\gamma)\cap
  N_{C}(\gamma))-\blambda(N_{B}(\gamma )\cap
  N_{C}(\gamma))}\right),
\end{align}
we claim that eq.~\ref{eq:fkg-1} is greater than
\begin{align}\label{eq:fkg-2}
  &\sum_{\substack{\mbf{e}\in \mathcal{P}(A\cup B\cup C) \\
  \mbf{d}\in\mathcal{P}(\bar{C})\\
  \sigma\in P(n+m+k,k)}}\prod_{i=1}^{n+m+2k}\left(\substack{\blambda(N_{A}(\mbf{e}_{1}^{\sigma_x(i)}))+\blambda(N_{B}(\mbf{e}_{1}^{\sigma_x(i)}))+\blambda(N_{C}(\mbf{e}_{1}^{\sigma_x(i)}))+\blambda(N_{\bar C}(\mbf{d}_{1}^{\sigma_y(i)}))\\-\blambda(N_{A}(\mbf{e}_{1}^{\sigma_x(i)})\cap
  N_{C}(\mbf{e}_{1}^{\sigma_x(i)}))-\blambda(N_{B}(\mbf{e}_{1}^{\sigma_x(i)})\cap
  N_{C}(\mbf{e}_{1}^{\sigma_x(i)}))+i\mu}\right)^{-1}.
\end{align}

Let $P(n,m,k,k)$ be the set of all increasing vertex paths from
$(0,0, 0,0)$ to $(n, m, k, k)$ in $\bbZ^4$, $\sigma(0)=(0,0,0,0)$ and
$\sigma(n+m+2k)=(n,k,k,k)$, for all $\sigma\in P(n,m,k,k)$.  We denote
the coordinates of $\sigma\in\bbZ^4$ by
$(\sigma_x,\sigma_y,\sigma_z,\sigma_w)$.  Using the canonical
bijection between
\begin{align*}
  &\cP(A\cup B\cup C\cup \bar C)\times P(n+m+k,k)\textrm{ and }\\
  &\cP(A)\times\cP(B)\times\cP(C)\times \cP(\bar C)\times P(n,m,k,k),
\end{align*}
we see that eq.~\ref{eq:fkg-2} is equal to
\begin{align*}
  \label{eq:fkg-3}\numberthis
  &\sum_{\substack{\mbf{a}\in
    \mathcal{P}(A),\mbf{b}\in\mathcal{P}(B)\\\mbf{c}\in\mathcal{P}(B),\mbf{\bar{c}}\in\mathcal{P}(\bar{C})\\\sigma\in
  P(n,m,k,k)}} \prod_{i=1}^{n+m+2k}\left(\substack{\blambda(N(\mbf{a}_{1}^{\sigma_x(i)}))+\blambda(N(\mbf{b}_{1}^{\sigma_y(i)}))+\blambda(N(\mbf{c}_{1}^{\sigma_z(i)}))+\blambda(N(\bar{\mbf{c}}_{1}^{\sigma_w(i)}))\\-\blambda(N(\mbf{a}_{1}^{\sigma_x(i)})\cap
  N(\mbf{c}_{1}^{\sigma_z(i)}))-\blambda(N(\mbf{b}_{1}^{\sigma_y(i)})\cap
  N(\mbf{c}_{1}^{\sigma_z(i)}))+i\mu}\right)^{-1}.
\end{align*}
Applying similar reductions to the RHS of eq.~\ref{eq:fkg-disjoint},
we note that the result follows if we prove
\begin{equation}
\begin{aligned}\label{eq:fkg-toprove}
  &\sum_{\substack{\mbf{a}\in \mathcal{P}(A),\mbf{b}\in\mathcal{P}(B)\\ \mbf{c}\in\mathcal{P}(B),\mbf{\bar{c}}\in\mathcal{P}(\bar{C})\\\sigma\in
  P(n,m,k,k)}}\prod_{i=1}^{n+m+2k}\left(\substack{\blambda(N(\mbf{a}_{1}^{\sigma_x(i)}))+\blambda(N(\mbf{b}_{1}^{\sigma_y(i)}))+\blambda(N(\mbf{c}_{1}^{\sigma_z(i)}))+\blambda(N(\bar{\mbf{c}}_{1}^{\sigma_w(i)}))\\-\blambda(N(\mbf{a}_{1}^{\sigma_x(i)})\cap
  N(\mbf{c}_{1}^{\sigma_z(i)}))-\blambda(N(\mbf{b}_{1}^{\sigma_y(i)})\cap
  N(\mbf{c}_{1}^{\sigma_z(i)}))+i\mu}\right)^{-1}\\
  &\geq \sum_{\substack{\mbf{a}\in \mathcal{P}(A),\mbf{b}\in\mathcal{P}(B)\\ \mbf{c}\in\mathcal{P}(B),\mbf{\bar{c}}\in\mathcal{P}(\bar{C})\\\sigma\in
  P(n,m,k,k)}}\prod_{i=1}^{n+m+2k}\left(\substack{\blambda(N(\mbf{a}_{1}^{\sigma_x(i)}))+\blambda(N(\mbf{b}_{1}^{\sigma_y(i)}))+\blambda(N(\mbf{c}_{1}^{\sigma_z(i)}))+\blambda(N(\bar{\mbf{c}}_{1}^{\sigma_w(i)}))\\-\blambda(N(\mbf{a}_{1}^{\sigma_x(i)})\cap
  N(\mbf{c}_{1}^{\sigma_z(i)}))-\blambda(N(\mbf{b}_{1}^{\sigma_y(i)})\cap
  N(\bar{\mbf{c}}_{1}^{\sigma_w(i)}))+i\mu}\right)^{-1}.
\end{aligned}
\end{equation}

Note that the only difference in the left and right sides of the last
inequality are the terms
$\blambda(N(\mbf{b}_{1}^{\sigma_y(i)})\cap
N(\bar{\mbf{c}}_{1}^{\sigma_z(i)}))$ and
$\blambda(N(\mbf{b}_{1}^{\sigma_y(i)})\cap
N(\bar{\mbf{c}}_{1}^{\sigma_w(i)}))$.

Equation~\ref{eq:fkg-toprove} can be expressed in the following
equivalent way
\begin{equation}
\begin{aligned} \label{eq:fkg-4}
  &\EE_{\mathfrak{Sabc\bar{c}}}\prod_{i=1}^{n+m+2k}\left(\substack{\blambda(N(\mathfrak{a}_{1}^{\mathfrak{S}_x(i)}))+\blambda(N(\mathfrak{b}_{1}^{\mathfrak{S}_y(i)}))+\blambda(N(\mathfrak{c}_{1}^{\mathfrak{S}_z(i)}))+\blambda(N(\bar{\mathfrak{c}}_{1}^{\mathfrak{S}_w(i)}))\\-\blambda(N(\mathfrak{a}_{1}^{\mathfrak{S}_x(i)})\cap
      N(\mathfrak{c}_{1}^{\mathfrak{S}_z(i)}))-\blambda(N(\mathfrak{b}_{1}^{\mathfrak{S}_y(i)})\cap
  N(\mathfrak{c}_{1}^{\mathfrak{S}_z(i)}))+i\mu}\right)^{-1}\\
  &\geq\EE_{\mathfrak{Sabc\bar{c}}}\prod_{i=1}^{n+m+2k}\left(\substack{\blambda(N(\mathfrak{a}_{1}^{\mathfrak{S}_x(i)}))+\blambda(N(\mathfrak{b}_{1}^{\mathfrak{S}_y(i)}))+\blambda(N(\mathfrak{c}_{1}^{\mathfrak{S}_z(i)}))+\blambda(N(\bar{\mathfrak{c}}_{1}^{\mathfrak{S}_w(i)}))\\-\blambda(N(\mathfrak{a}_{1}^{\mathfrak{S}_x(i)})\cap
  N(\mathfrak{c}_{1}^{\mathfrak{S}_z(i)}))-\blambda(N(\mathfrak{b}_{1}^{\mathfrak{S}_y(i)})\cap
  N(\bar{\mathfrak{c}}_{1}^{\mathfrak{S}_w(i)}))+i\mu}\right)^{-1},
\end{aligned}
\end{equation}
where the expectation is over a uniformly random element
$(\mathfrak{a},\mathfrak{b},\mathfrak{c},\mathfrak{\bar{c}},\mathfrak{S})$
of the set
$\mathcal{P}(A)\times\mathcal{P}(B)\times\mathcal{P}(C)\times
\mathcal{P}(\bar C)\times P(n,m,k,k)=\mathcal{P}(A\cup B\cup
C\cup\bar{C})$.  In the following, we will simply write $\EE$ in place
of the symbol $\EE_{\mathfrak{Sabc\bar{c}}}$.

To prove eq.~\ref{eq:fkg-4}, we first prove it on a smaller
$\sigma$-algebra.  We say that two permutations $\gamma_1$ and
$\gamma_2\in\cP(A\cup B\cup C\cup \bar{C})$ are equivalent if by
dropping the overline marks of the particles in $C$ in both $\gamma_1$
and $\gamma_2$, we obtain the same sequence of elements.  Let $\cF$
denote the $\sigma$-algebra generated by this equivalence relation.
We show that for any
$(\sigma,\mbf{a},\mbf{b},\mbf{c},\mbf{\bar{c}})\in \mathcal{P}(A\cup
B\cup C\cup\bar C)$,
\begin{equation}
\begin{aligned}\label{eq:fkg-cond-ineq}
  & \EE\left[\prod_{i=1}^{n+m+2k}\left(\substack{\blambda(N(\mathfrak{a}_{1}^{\mathfrak{S}_x(i)}))+\blambda(N(\mathfrak{b}_{1}^{\mathfrak{S}_y(i)}))+\blambda(N(\mathfrak{c}_{1}^{\mathfrak{S}_z(i)}))+\blambda(N(\bar{\mathfrak{c}}_{1}^{\mathfrak{S}_w(i)}))\\-\blambda(N(\mathfrak{a}_{1}^{\mathfrak{S}_x(i)})\cap
  N(\mathfrak{c}_{1}^{\mathfrak{S}_z(i)}))-\blambda(N(\mathfrak{b}_{1}^{\mathfrak{S}_y(i)})\cap
  N(\mathfrak{c}_{1}^{\mathfrak{S}_z(i)}))+i\mu}\right)^{-1}\bigg|\cF\right](\sigma,\mbf{a},\mbf{b},\mbf{c},\mbf{\bar{c}})\\
  &\geq\EE\left[\prod_{i=1}^{n+m+2k}\left(\substack{\blambda(N(\mathfrak{a}_{1}^{\mathfrak{S}_x(i)}))+\blambda(N(\mathfrak{b}_{1}^{\mathfrak{S}_y(i)}))+\blambda(N(\mathfrak{c}_{1}^{\mathfrak{S}_z(i)}))+\blambda(N(\bar{\mathfrak{c}}_{1}^{\mathfrak{S}_w(i)}))\\-\blambda(N(\mathfrak{a}_{1}^{\mathfrak{S}_x(i)})\cap
  N(\mathfrak{c}_{1}^{\mathfrak{S}_z(i)}))-\blambda(N(\mathfrak{b}_{1}^{\mathfrak{S}_y(i)})\cap
  N(\bar{\mathfrak{c}}_{1}^{\mathfrak{S}_w(i)}))+i\mu}\right)^{-1}\bigg|\cF\right] (\sigma,\mbf{a},\mbf{b},\mbf{c},\mbf{\bar{c}}).
\end{aligned}
\end{equation}
Fix
$(\sigma,\mbf{abc\bar{c}})\in \mathcal{P}(A\cup B\cup C\cup\bar C)$.
We can express $\mbf{\bar{c}}$ as a composition of a permutation
$\tau\in \mathcal{P}([k])$ and $\mbf{c}$, so that
$\bar{c}_i=c_{\tau(i)}$.  Since all permutations in the equivalence
class of $(\sigma,\mbf{a},\mbf{b},\mbf{c},\mbf{\bar{c}})$ are equally
likely, each conditional expectation in eq.~\ref{eq:fkg-cond-ineq} can
be expressed as an expectation over auxilliary i.i.d.
Bernoulli$(1/2)$ random variables $\{\beta_i\}_{i=1}^k$.  To make this
precise, let $S_{i,j}=\1(j\leq\sigma_z(i))$,
$T_{i,j}=\1(\tau^{-1}(j)\leq \sigma_w(i))$,
$p_{i}=i\mu+\blambda(N(\mbf{a}_1^{\sigma_x(i)}))+\blambda(N(\mbf{b}_1^{\sigma_y(i)}))$
and $\bar\beta_j=1-\beta_j$ for all $1\leq i\leq n+m+2k$ and
$1\leq j\leq k$.  Also, let
$\mbf{c}^{\beta\sigma(i)}=\{c_j\in C:S_{i,j}=1,
\beta_j=1\}\cup\{c_j\in C:T_{i,j}=1,\beta_j=0\}$ and similarly,
$\mbf{c}^{\bar \beta\sigma(i)}=\{c_j\in
C:S_{i,j}=1,\bar\beta_j=1\}\cup\{c_j\in
C:T_{i,j}=1,\bar\beta_j=0\}$. We have
\begin{equation}
\begin{aligned}\label{eq:fkg-beta}
  &\EE_\beta\left[\prod_{i=1}^{n+m+2k}\left(\substack{p_i+\blambda(N(\mbf{c}^{\beta\sigma(i)}))+\blambda(N(\mbf{c}^{\bar\beta\sigma(i)}))\\-\blambda(N(\mbf{c}^{\beta\sigma(i)})\cap
  N(\mbf{a}_1^{\sigma_x(i)}))-\blambda(N(\mbf{c}^{\beta\sigma(i)})\cap
  N(\mbf{b}_1^{\sigma_y(i)}))}\right)^{-1}\right.\\
  &\IndI\left.-\prod_{i=1}^{n+m+2k}\left(\substack{p_i+\blambda(N(\mbf{c}^{\beta\sigma(i)}))+\blambda(N(\mbf{c}^{\bar\beta\sigma(i)}))\\-\blambda(N(\mbf{c}^{\beta\sigma(i)})\cap
  N(\mbf{a}_1^{\sigma_x(i)}))-\blambda(N(\mbf{c}^{\bar\beta\sigma(i)})\cap
  N(\mbf{b}_1^{\sigma_y(i)}))}\right)^{-1}\right]\geq 0,
\end{aligned}
\end{equation}
Using the fact that $\int_0^\infty e^{-cx}dx=\frac{1}{c}$ for any
$c>0$, we may write the above inequality as
\begin{equation}
\begin{aligned}\label{eq:fkg-beta-2}
  &\int_{{\bbR^+}^{n+m+2k}}\EE_\beta\left[\exp\left(-\sum_{i=1}^{n+m+2k}x_i\left(\substack{p_i+\blambda(N(\mbf{c}^{\beta\sigma(i)}))+\blambda(N(\mbf{c}^{\bar\beta\sigma(i)}))\\-\blambda(N(\mbf{c}^{\beta\sigma(i)})\cap
  N(\mbf{a}_1^{\sigma_x(i)}))-\blambda(N(\mbf{c}^{\beta\sigma(i)})\cap
  N(\mbf{b}_1^{\sigma_y(i)}))}\right)\right)\right.\\
  &\IndI\left.-\exp\left(-\sum_{i=1}^{n+m+2k}x_i\left(\substack{p_i+\blambda(N(\mbf{c}^{\beta\sigma(i)}))+\blambda(N(\mbf{c}^{\bar\beta\sigma(i)}))\\-\blambda(N(\mbf{c}^{\beta\sigma(i)})\cap
  N(\mbf{a}_1^{\sigma_x(i)}))-\blambda(N(\mbf{c}^{\bar\beta\sigma(i)})\cap
  N(\mbf{b}_1^{\sigma_y(i)}))}\right)\right)\right]d\mbf{x}_1^{n+m+2k}\geq 0.
\end{aligned}
\end{equation}

It is enough to prove that the integrand in positive for every
$\mbf{x}_1^{n+m+2k}$.  Symmetrizing the expression, by replacing
$\beta_j$ with $\bar\beta_j$, we obtain the following equivalent
expression.
\begin{equation}
\begin{aligned}\label{eq:fkg-beta-3}
  0&\leq \EE_\beta\left[\exp\left(-\sum_{i=1}^{n+m+2k}x_i\blambda(N(\mbf{c}^{\beta\sigma(i)}))+x_i\blambda(N(\mbf{c}^{\bar\beta\sigma(i)}))\right)\right.\\
   &\IndI\left.\times\left(
     \substack{\exp\left(\sum_{i}x_i\blambda(N(\mbf{c}^{\beta\sigma(i)})\cap
     N(\mbf{a}_1^{\sigma_x(i)}))\right)-\exp\left(\sum_{i}x_i\blambda(N(\mbf{c}^{\bar\beta\sigma(i)})\cap
     N(\mbf{a}_1^{\sigma_x(i)}))\right)}\right)\right.\\
   &\IndI\left.\times\left(
     \substack{\exp\left(\sum_{i}x_i\blambda(N(\mbf{c}^{\beta\sigma(i)})\cap
     N(\mbf{b}_1^{\sigma_y(i)}))\right)-\exp\left(\sum_{i}x_i\blambda(N(\mbf{c}^{\bar\beta\sigma(i)})\cap
     N(\mbf{b}_1^{\sigma_y(i)}))\right)}\right)\vphantom{\sum_1^n}\right].
\end{aligned}
\end{equation}
To prove this, we use the FKG inequality on the lattice
$\{0,1\}^{n+m+2k}$ with measure
\begin{align*}
  \nu(\mbf{\beta})=\exp\left(-\sum_{i=1}^{n+m+2k}x_i\blambda(N(\mbf{c}^{\beta\sigma(i)}))+x_i\blambda(N(\mbf{c}^{\bar\beta\sigma(i)}))\right).
\end{align*}

\begin{claim}\label{claim:fkg-one}
  The measure $\nu$ is log-submodular.
\end{claim}
\begin{proof}
  Let $\mbf{\beta},\mbf{\gamma}\in\{0,1\}^{n+m+2k}$.  Then,
  \begin{equation}
  \begin{aligned}
    &\left(\substack{\blambda
      (N(\{c_j:S_{i,j}=1,\beta_j\vee\gamma_j=1\}\cup\{c_j:T_{i,j}=1,\beta_j\vee\gamma_j=0\}))-\blambda (N(\{c_j:S_{i,j}=1,\beta_j=1\}\cup\{c_{j}:T_{i,j}=1,\beta_j=0\}))\\
    -\blambda (N(\{c_j:S_{i,j}=1,\gamma_j=1\}\cup\{c_{j}:T_{i,j}=1,\gamma_j=0\}))+\blambda (N(\{c_j:S_{i,j}=1,\beta_j\wedge\gamma_j=1\}\cup\{c_j:T_{i,j}=1,\beta_j\wedge\gamma_j=0\}))}\right)\\
    &=\left(\substack{\blambda [N(\{c_j:S_{i,j}=1,\beta_j\vee\gamma_j=1\})]-\blambda [N(\{c_j:S_{i,j}=1,\beta_j=1\})]\\
    -\blambda [N(\{c_j:S_{i,j}=1,\gamma_j=1\})]+\blambda
    [N(\{c_j:S_{i,j}=1,\beta_j\wedge\gamma_j=1\})]}\right)\\
    &\IndII+\left(\substack{\blambda [N(\{c_j:T_{i,j}=1,\beta_j\vee\gamma_j=0\})]-\blambda [N(\{c_{j}:T_{i,j}=1,\beta_j=0\})]\\
    -\blambda [N(\{c_{j}:T_{i,j}=1,\gamma_j=0\})]+\blambda [N(\{c_j:T_{i,j}=1,\beta_j\wedge\gamma_j=0\})]}\right)\label{eq:submod-1}\\
    &\IndII-\left(\substack{\blambda [N(\{c_j:S_{i,j}=1,\beta_j\vee\gamma_j=1\})\cap
      N(\{c_j:T_{i,j}=1,\beta_j\vee\gamma_j=0\})]\\-\blambda [N(\{c_j:S_{i,j}=1,\beta_j=1\})\cap
      N(\{c_{j}:T_{i,j}=1,\beta_j=0\})]\\
    -\blambda [N(\{c_j:S_{i,j}=1,\gamma_j=1\})\cap
    N(\{c_{j}:T_{i,j}=1,\gamma_j=0\})]\\+\blambda [N(\{c_j:S_{i,j}=1,\beta_j\wedge\gamma_j=1\})\cap
    N(\{c_j:T_{i,j}=1,\beta_j\wedge\gamma_j=0\})]}\right).
  \end{aligned}
\end{equation}

  Let us look at the first term in eq.~\ref{eq:submod-1}.  We have
  \begin{equation*}
  \begin{aligned}
    &\left(\substack{\blambda [N(\{c_j:S_{i,j}=1,\beta_j\vee\gamma_j=1\})]-\blambda [N(\{c_j:S_{i,j}=1,\beta_j=1\})]\\
    -\blambda [N(\{c_j:S_{i,j}=1,\gamma_j=1\})]+\blambda
    [N(\{c_j:S_{i,j}=1,\beta_j\wedge\gamma_j=1\})]}\right) \\
    &= \left(\substack{\blambda
      [N(\{c_j:S_{i,j}=1,\beta_j\wedge\gamma_j=1\})]\\-\blambda
    [N(\{c_j:S_{i,j}=1,\beta_j=1\})\cap N(\{c_j:S_{i,j}=1,\gamma_j=1\})]}\right),
  \end{aligned}
\end{equation*}
which is non-positive since
$N(\{c_j:S_{i,j}=1,\beta_j\wedge\gamma_j=1\})$ is contained in both
$N(\{c_j:S_{i,j}=1,\beta_j=1\})$ and
$N(\{c_j:S_{i,j}=1,\gamma_j=1\})$.

  Similarly, we may prove that the second term in
  eq.~\ref{eq:submod-1} is non-positive.  For the third term in that
  equation, we have:
  \begin{align*}
    &-\left(\substack{-\blambda [N(\{c_j:S_{i,j}=1,\gamma_j=1\})\cap
      N(\{c_{j}:T_{i,j}=1,\gamma_j=0\})]\\-\blambda [N(\{c_j:S_{i,j}=1,\beta_j=1\})\cap
      N(\{c_{j}:T_{i,j}=1,\beta_j=0\})]\\
    +\blambda [N(\{c_j:S_{i,j}=1,\beta_j\vee\gamma_j=1\})\cap
    N(\{c_j:T_{i,j}=1,\beta_j\vee\gamma_j=0\})]\\+\blambda N(\{c_j:S_{i,j}=1,\beta_j\wedge\gamma_j=1\})\cap
    N(\{c_j:T_{i,j}=1,\beta_j\wedge\gamma_j=0\})}\right)\\
    &\leq -\left(\substack{-\blambda [N(\{c_j:S_{i,j}=1,\gamma_j=1\})\cap
      N(\{c_{j}:T_{i,j}=1,\beta_j\vee\gamma_j=0\})]\\-\blambda [N(\{c_j:S_{i,j}=1,\beta_j\wedge\gamma_j=1\})\cap
      N(\{c_{j}:T_{i,j}=1,\beta_j=0\})]\\
    +\blambda [N(\{c_j:S_{i,j}=1,\beta_j\vee\gamma_j=1\})\cap
    N(\{c_j:T_{i,j}=1,\beta_j\vee\gamma_j=0\})]\\+\blambda [N(\{c_j:S_{i,j}=1,\beta_j\wedge\gamma_j=1\})\cap
    N(\{c_j:T_{i,j}=1,\beta_j\wedge\gamma_j=0\})]}\right)\\
    &\leq 0.
  \end{align*}

  By symmetry,
  \begin{align*}
    &\left(\substack{\blambda N(\{c_j:S_{i,j}=1,\beta_j\vee\gamma_j=0\}\cup\{c_j:T_{i,j}=1,\beta_j\vee\gamma_j=1\})-\blambda N(\{c_j:S_{i,j}=1,\beta_j=0\}\cup\{c_{j}:T_{i,j}=1,\beta_j=1\})\\
    -\blambda
    N(\{c_j:S_{i,j}=1,\gamma_j=0\}\cup\{c_{j}:T_{i,j}=1,\gamma_j=1\})+\blambda
    N(\{c_j:S_{i,j}=1,\beta_j\wedge\gamma_j=0\}\cup\{c_j:T_{i,j}=1,\beta_j\wedge\gamma_j=1\})}\right)\leq 0.
  \end{align*}
  Therefore,
  \begin{align*}
    &\sum_{i}^{}-x_i\blambda(N(\mbf{c}^{\beta\vee\gamma\sigma(i)}))-x_i\blambda(N(\mbf{c}^{\bar{\beta\vee\gamma}\sigma(i)})-\sum_{i}^{}x_i\blambda(N(\mbf{c}^{\beta\wedge\gamma\sigma(i)}))+x_i\blambda(N(\mbf{c}^{\bar{\beta\wedge\gamma}\sigma(i)})\\
    &+\sum_{i}^{}x_i\blambda(N(\mbf{c}^{\beta\sigma(i)}))+x_i\blambda(N(\mbf{c}^{\bar{\beta}\sigma(i)})+\sum_{i}^{}x_i\blambda(N(\mbf{c}^{\gamma\sigma(i)}))+x_i\blambda(N(\mbf{c}^{\bar{\gamma}\sigma(i)})\\
    &\geq 0.
  \end{align*}
  Consequently,
  $\nu(\mbf{\beta}\vee\mbf{\gamma})\nu(\mbf{\beta}\wedge\mbf{\gamma})\geq
  \nu(\mbf{\beta})\nu(\mbf{\gamma})$.
\end{proof}

Now we show that the two relevant functions in eq.~\ref{eq:fkg-beta-3}
are increasing in $\beta$.

\begin{claim}\label{claim:fkg-two}
  The functions
  \begin{align*}
    f(\mbf{\beta})&=\exp\left(\sum_{i}x_i\blambda(N(\mbf{c}^{\beta\sigma(i)})\cap
      N(\mbf{a}_1^{\sigma_x(i)}))\right)\\
    &\IndI-\exp\left(\sum_{i}x_i\blambda(N(\mbf{c}^{\bar\beta\sigma(i)})\cap
                    N(\mbf{a}_1^{\sigma_x(i)}))\right),\\
    \textrm{and }g(\mbf{\beta}) &=\exp\left(\sum_{i}x_i\blambda(N(\mbf{c}^{\beta\sigma(i)})\cap
                                  N(\mbf{b}_1^{\sigma_y(i)}))\right)\\
    &\IndI-\exp\left(\sum_{i}x_i\blambda(N(\mbf{c}^{\bar\beta\sigma(i)})\cap
                                 N(\mbf{b}_1^{\sigma_y(i)}))\right)
  \end{align*}
  are increasing in $\beta$.
\end{claim}
\begin{proof}{}
  Let
  $h(\mbf{\beta})=\sum_{i}x_i\blambda(N(c^{\beta\sigma(i)})\cap
  N(\mbf{a}_1^{\sigma_x(i)}))$.  For any $J\subseteq[k]$, let
  $q_{i,J}=\blambda(\cap_{j\in J}N(c_j)\cap
  N(\mbf{a}_1^{\sigma_x(i)}))$.  Using the inclusion-exclusion
  formula, we may write
  \begin{align*}
    h(\mbf{\beta})=\sum_{i}x_i\sum_{J\subseteq[k]}(-1)^{|J|-1}q_{i,J}\prod_{j\in
    J}(\beta_jS_{i,j}+\bar\beta_jT_{i,j}).
  \end{align*}

  Now, let $\beta_1,\ldots,\beta_{k}$ be given.  Fix $l\in [k]$.
  Fixing all $\beta_j$, $j\neq l$ and taking the difference of the
  values of $h$ when $\beta_l=1$ and $\beta_l=0$, we obtain:
  \begin{align*}
    &h(\mbf{\beta},\beta_l=1)-h(\mbf{\beta},\beta_l=0)\\
    &=\sum_{i=1}^{n+m+2k}x_i(S_{i,l}-T_{i,l})\sum_{J\subseteq[k-1]}(-1)^{|J|}q_{i,\{Jl\}}\prod_{j\in
      J}(\beta_jS_{i,j}+\bar\beta_jT_{i,j})\\
    &=\sum_{i=1}^{n+m+2k}x_i(S_{i,l}-T_{i,l})\blambda[(N(c_l)\cap
      N(\mbf{a}_1^{\sigma_x(i)}))\backslash
      N(c_1^{\sigma_z(i)}\backslash c_l)]\\
    &\geq  0,
  \end{align*}
  since we have assumed that $S_{i,l}\geq T_{i,l}$.  Similarly, taking
  \begin{align*}
    h'(\mbf{x})&=\sum_ix_i\blambda(N(c^{\beta\sigma(i)})\cap
                 N(\mbf{a}_1^{\sigma_x(i)}))\\
               &=\sum_{i}x_i\sum_{J\subseteq[k]}(-1)^{|J|-1}q_{i,J}\prod_{j\in
                 J}(\bar\beta_jS_{i,j}+\beta_jT_{i,j}),
  \end{align*}
  we have
  \begin{align*}
    &h'(\mbf{\beta},\beta_l=1)-h'(\mbf{\beta},\beta_l=0)\\
    &=\sum_{i=1}^{n+m+2k}x_i(T_{i,l}-S_{i,l})\blambda[(N(c_l)\cap
      N(\mbf{a}_1^{\sigma_x(i)}))\backslash
      N(c_1^{\sigma_w(i)}\backslash c_l)]\\
    &\leq  0 .
  \end{align*}
  Thus, $f(\mbf{\beta},\beta_l=1)-f(\mbf{\beta},\beta_l=0)\geq 0$.  By
  symmetry in the problem, this is also true for $g$.
\end{proof}

We are now in a position to apply the FKG theorem to the RHS of
eq.~\ref{eq:fkg-beta-3}, and since
\begin{align*}
  & \EE_\beta\left[e^{-\sum_{i=1}^{n+m+2k}x_i\blambda(N(\mbf{c}^{\beta\sigma(i)}))+x_i\blambda(N(\mbf{c}^{\bar\beta\sigma(i)}))}\right.\\
  &\Ind\left.\times\left(
    e^{\sum_{i}x_i\blambda(N(\mbf{c}^{\beta\sigma(i)})\cap
    N(\mbf{a}_1^{\sigma_x(i)}))}-e^{\sum_{i}x_i\blambda(N(\mbf{c}^{\bar\beta\sigma(i)})\cap
    N(\mbf{a}_1^{\sigma_x(i)}))}\right)\right]\\
  &=0,
\end{align*}
we obtain the result.

\clearpage
\section{Table of Notation}
\label{sec:tabl-nota}
\renewcommand{\arraystretch}{1.2}
\hypersetup{hidelinks} \rowcolors{1}{}{gray!20}
\begin{longtable}{|l|m{0.6\textwidth}|} \hline
\hyperlink{def:formaldesc}{$D$} & Domain of interaction of particles.
A metric space \\ \hyperlink{def:colorset}{$\mbf{C}:=\{\ttR, \ttB\}$}
  & The set of types of particles, \emph{reds} and \emph{blues} \\
\hyperlink{def:oppositecolor}{$\bar{\ttR}:=\ttB$, $\bar{\ttB}:=\ttR$}
  & Opposite color\\ \hyperlink{def:formaldesc}{$\lambda$}
  & Radon measure on $D$\\ \hyperlink{def:formaldesc}{$m_{\mbf{C}}$}
  & Counting measure on $\mbf{C}$\\
\hyperlink{def:formaldesc}{$\blambda$}
  & $\blambda:=\lambda\otimes m_{\mbf{C}}$\\
\hyperlink{def:formaldesc}{$\mu$}
  & The parameter of the exponential random variables describing
patience of particles.\\ \hyperlink{def:notation}{$M(D,K)$}
  & Space of simple Radon counting measures on $D$, with marks in
$K$\\ \hyperlink{def:notation}{$O(D,K)$}
  & Space of simple locally-finite ordered subsets of $D$, with marks
in $K$ \\ \hyperlink{def:notation}{$|\gamma|$}, $\gamma\in O(D,K)$
  & Number of elements in $\gamma$\\
\hyperlink{def:gamma-sup-x}{$\gamma^x$}
  &The set $\{y\in\gamma:y<_\gamma x\}$ ordered as in $\gamma$ \\
\hyperlink{def:p-x-b-x}{$p_x,\ b_x,\ c_x,\ w_x$}, $x\in
D\times\mbf{C}$
  & Position, birth time, color and patience of $x$, i.e.,
$x=(p_x,c_x)$.\\ \hyperlink{def:N}{$N(A)$}, $A\subset D\times\mbf{C}$
  & $N(A):=\cup_{x\in A}\{y\in D\times\mbf{C}:c_y\neq c_x,
d(p_y,p_x)<1 \}$.\\ \hyperlink{def:N}{$W_x$}
  & Region of maximum priority of $x\in\gamma$.  $W_x=N(x)\backslash
N(\gamma_x)$\\ \hyperlink{def:formaldesc}{$\eta_t\in O(D,\mbf{C})$}
  & Ordered collection of particles present in the system at time
$t$\\ \hyperlink{def:Phi}{$\Phi$}
  & Poisson arrival process used in the construction of the process.
It is a random element of $M(D\times \bbR^+,\mbf{C}\times\bbR^+)$\\
\hyperlink{def:specials}{$S_t$}
  & Set of discrepancies $\eta^0_t\triangle \eta^1_t$ in the CFTP
construction\\ \hyperlink{def:kappa}{$\kappa$}
  & Killing function\\ \hyperlink{def:matching-func}{$m$}
  & Matching function\\ \hyperlink{def:matching-marks}{$\ttu$, $\ttm$}
  & Marks to indicate whether a particle is matched or unmatched in
    the detailed processes\\
  \hyperlink{def:hat-eta_t}{$\hat\eta_t$}
  & Backward detailed process\\ \hyperlink{def:check-eta_t}{$\check\eta_t$}
  & Forward detailed process\\
\hyperlink{def:Q-i-u}{$Q^i_\ttu(\gamma)$}
  & For $\gamma\in O(D,\mbf{C}\times\{\ttu,\ttm\})$, it is the number
of unmatched particles among the first $i$ particles of $\gamma$.  \\
\hyperlink{def:Q-i-m}{$Q^i_\ttm(\gamma)$}
  & For $\gamma\in O(D,\mbf{C}\times\{\ttu,\ttm\})$, it is the number
of matched particles excluding the first $i$ particles of $\gamma$.
\\ \hyperlink{def:hat-pi}{$\hat\pi$}
  & Density of the stationary measure of the Backward detailed
    process\\
  \hyperlink{def:pi}{$\pi$}& Density of the stationary
measure of the process $\eta_t$\\
\hyperlink{def:tilde-pi}{$\tilde\pi$}& Janossy density of stationary
version of the point process $\eta_0$\\
\hyperlink{def:P(C)}{$\mathcal{P}(C)$} & Set of all permutations of
the elements of a finite set $C$.  \\
\hyperlink{def:paths-space}{$P(m,n)$}& The set of all paths in a
square lattice from $(0,0)$ to $(m,n)$\\
\hyperlink{def:order-bijection}{$(\sigma,X_1^n,Y_1^m)$} & A
representation of the map that gives the canonical bijection between
$P(n,m)\times \mathcal{P}(x_1^n)\times\mathcal{P}(y_1^m)$ and
$\mathcal{P}(x_1^n,y_1^m)$.  \\ \hline
\end{longtable}

\section*{Acknowledgements}
\label{sec:acknowledgements}
The author would like to thank his PhD advisor, Prof. Fran\c{c}ois
Baccelli, for many valuable discussion on this problem.

\bibliographystyle{abbrv}
\bibliography{bibliography} 
\end{document}

%% file: figure1.pdf_tex
\begingroup%
  \makeatletter%
  \providecommand\color[2][]{%
    \errmessage{(Inkscape) Color is used for the text in Inkscape, but the package 'color.sty' is not loaded}%
    \renewcommand\color[2][]{}%
  }%
  \providecommand\transparent[1]{%
    \errmessage{(Inkscape) Transparency is used (non-zero) for the text in Inkscape, but the package 'transparent.sty' is not loaded}%
    \renewcommand\transparent[1]{}%
  }%
  \providecommand\rotatebox[2]{#2}%
  \newcommand*\fsize{\dimexpr\f@size pt\relax}%
  \newcommand*\lineheight[1]{\fontsize{\fsize}{#1\fsize}\selectfont}%
  \ifx\svgwidth\undefined%
    \setlength{\unitlength}{453.54330709bp}%
    \ifx\svgscale\undefined%
      \relax%
    \else%
      \setlength{\unitlength}{\unitlength * \real{\svgscale}}%
    \fi%
  \else%
    \setlength{\unitlength}{\svgwidth}%
  \fi%
  \global\let\svgwidth\undefined%
  \global\let\svgscale\undefined%
  \makeatother%
  \begin{picture}(1,0.61875)%
    \lineheight{1}%
    \setlength\tabcolsep{0pt}%
    \put(0,0){\includegraphics[width=\unitlength,page=1]{figure1.pdf}}%
  \end{picture}%
\endgroup%

%% file: matching_dynamics_FCFS_v3.bbl
\begin{thebibliography}{10}

\bibitem{Adan2018}
I.~Adan, A.~Bu{\v{s}}i{\'{c}}, J.~Mairesse, and G.~Weiss.
\newblock {Reversibility and Further Properties of FCFS Infinite Bipartite
  Matching}.
\newblock {\em Mathematics of Operations Research}, 43(2):598--621, 5 2018.

\bibitem{adan2012exact}
I.~Adan and G.~Weiss.
\newblock Exact fcfs matching rates for two infinite multitype sequences.
\newblock {\em Operations research}, 60(2):475--489, 2012.

\bibitem{baccelli2017mutual}
F.~Baccelli, F.~Mathieu, and I.~Norros.
\newblock Mutual service processes in euclidean spaces: existence and
  ergodicity.
\newblock {\em Queueing Systems}, 86(1-2):95--140, 2017.

\bibitem{blaszczyszyn2014comparison}
B.~B{\l}aszczyszyn and D.~Yogeshwaran.
\newblock On comparison of clustering properties of point processes.
\newblock {\em Advances in Applied Probability}, 46(1):1--20, 2014.

\bibitem{buke2017fluid}
B.~B{\"u}ke and H.~Chen.
\newblock Fluid and diffusion approximations of probabilistic matching systems.
\newblock {\em Queueing Systems}, 86(1-2):1--33, 2017.

\bibitem{buvsic2013stability}
A.~Bu{\v{s}}i{\'c}, V.~Gupta, and J.~Mairesse.
\newblock Stability of the bipartite matching model.
\newblock {\em Advances in Applied Probability}, 45(2):351--378, 2013.

\bibitem{caldentey2009fcfs}
R.~Caldentey, E.~H. Kaplan, and G.~Weiss.
\newblock Fcfs infinite bipartite matching of servers and customers.
\newblock {\em Advances in Applied Probability}, 41(3):695--730, 2009.

\bibitem{chayes1995analysis}
J.~Chayes, L.~Chayes, and R.~Kotecky.
\newblock The analysis of the widom-rowlinson model by stochastic geometric
  methods.
\newblock {\em Communications in Mathematical Physics}, 172(3):551--569, 1995.

\bibitem{daley2007introduction}
D.~J. Daley and D.~Vere-Jones.
\newblock {\em An introduction to the theory of point processes: volume II:
  general theory and structure}, volume~2.
\newblock Springer Science \& Business Media, 2007.

\bibitem{engel1999one}
K.-J. Engel and R.~Nagel.
\newblock {\em One-parameter semigroups for linear evolution equations}, volume
  194.
\newblock Springer Science \& Business Media, 1999.

\bibitem{georgii1997stochastic}
H.-O. Georgii and T.~K{\"u}neth.
\newblock Stochastic comparison of point random fields.
\newblock {\em Journal of Applied Probability}, 34(4):868--881, 1997.

\bibitem{giacomin1995agreement}
G.~Giacomin, J.~Lebowitz, and C.~Maes.
\newblock Agreement percolation and phase coexistence in some gibbs systems.
\newblock {\em Journal of statistical physics}, 80(5-6):1379--1403, 1995.

\bibitem{holley1974remarks}
R.~Holley.
\newblock Remarks on the fkg inequalities.
\newblock {\em Communications in Mathematical Physics}, 36(3):227--231, 1974.

\bibitem{kashyap1966double}
B.~R. Kashyap.
\newblock The double-ended queue with bulk service and limited waiting space.
\newblock {\em Operations Research}, 14(5):822--834, 1966.

\bibitem{kelly2011reversibility}
F.~P. Kelly.
\newblock {\em Reversibility and stochastic networks}.
\newblock Cambridge University Press, 2011.

\bibitem{meester1996continuum}
R.~Meester and R.~Roy.
\newblock {\em Continuum percolation}, volume 119.
\newblock Cambridge University Press, 1996.

\bibitem{ruelle1971existence}
D.~Ruelle.
\newblock Existence of a phase transition in a continuous classical system.
\newblock {\em Physical Review Letters}, 27(16):1040, 1971.

\bibitem{van2000markov}
M.~Van~Lieshout.
\newblock {\em Markov point processes and their applications}.
\newblock World Scientific, 2000.

\bibitem{widom1970new}
B.~Widom and J.~S. Rowlinson.
\newblock New model for the study of liquid--vapor phase transitions.
\newblock {\em The Journal of Chemical Physics}, 52(4):1670--1684, 1970.

\end{thebibliography}
